\documentclass{amsart}
\usepackage{amsmath,amssymb,amsthm,amscd}

\def\p{\partial}
\def\R{\mathbb{R}}
\def\C{\mathbb{C}}

\def\Z{\mathbb{Z}}

\def\l{\lambda}

\def\i{\sqrt{-1}}
	
\def\t{\triangle}

\def\cD{\mathcal D}

\def\cF{{\mathcal F}}

\def\cH{{\mathcal H}}

\def\cO{{\mathcal O}}

\def\cW{{\mathcal W}}

\numberwithin{equation}{section}

\newtheorem{prop}{Proposition}[section]
\newtheorem{theo}[prop]{Theorem}
\newtheorem{lemma}[prop]{Lemma}

\newtheorem{rmk}[prop]{Remark}

\newtheorem{defi}[prop]{Definition}

\newtheorem{pr}[prop]{Problem}
\newtheorem{claim}[prop]{Claim}

\begin{document}
\title[Sasakian geometry and the Sasaki-Ricci flow]{The Sasaki-Ricci flow and Compact Sasakian manifolds 
of positive  transverse holomorphic bisectional curvature}

\author{Weiyong He}

\address{Department of Mathematics, University of Oregon, Eugene, Oregon, 97403}
\email{whe@uoregon.edu}
\thanks{The author is partially supported by a NSF grant.}
\date{}
\maketitle

\section{Introduction}

Sasakian geometry is the odd dimensional cousin of the K\"ahler geometry. 
Perhaps the most straightforward definition is the following: a Riemannian manifold $(M, g)$ is Sasakian if and only if its metric cone $(C(M)=\R^{+}\times M, \bar g=dr^2+r^2 g)$ is K\"ahler. A Sasakian-Einstein manifold is a Riemannian manifold that is both Sasakian and Einstein. There has been renewed extensive interest recently on Sasakian geometry, especially on Sasakian-Einstein manifolds. Readers are referred to recent monograph Boyer-Galicki \cite{BG}, and recent survey paper Sparks \cite{Sparks} and the references in for history, background and recent progress of Sasakian geometry and Sasaki-Einstein manifolds. 

The {\it Sasaki-Ricci flow} is introduced by Smoczyk-Wang-Zhang  \cite{SWZ} to study the existence of Sasaki-Einstein metrics, more precisely, the {\it $\eta$-Einstein metrics} on Sasakian manifolds. On a Sasakian manifold, there is a  one-dimensional foliation structure $\cF_\xi$ determined by the {\it Reeb vector filed} $\xi$. The transverse structure of the foliation is a {\it transverse K\"ahler structure}. And the Sasaki-Ricci flow is just a {\it transverse K\"ahler-Ricci flow} which deforms the transverse K\"ahler structure along its (negative) transverse Ricci curvature on Sasakian manifolds. In particular, short-time and long time existence of the Sasaki-Ricci flow are proved  and convergence to an $\eta$-Einstein metric is also established when the {\it basic first  Chern class} is negative or null \cite{SWZ}, which can be viewed as an odd-dimensional counterpart of Cao's result \cite{Cao} for the K\"ahler-Ricci flow. 

For Hamilton's Ricci flow, Perelman \cite{Perelman01} introduced a functional, called the $\cW$ functional, which is monotone along the Ricci flow and has tremendous applications in  the Ricci flow. As an application to K\"ahler geometry, Perelman proved  deep results  in K\"ahler-Ricci flow when the first Chern class is positive. For example, he proved that the scalar curvature and the diameter are both uniformly bounded; readers are referred to Sesum-Tian \cite{Sesum-Tian} for details. Perelman's results strengthen the belief of Hamilton-Tian conjecture which states that the K\"ahler-Ricci flow would converge  in some fashion when the first Chern class is positive.

In this paper we will study the Sasaki-Ricci flow when the first basic Chern class is positive. First of all, we introduce an analogue of Perelman's $\cW$ functional on Sasakian manifolds, which is monotone along the Sasaki-Ricci flow. Then as in the K\"ahler setting \cite{Sesum-Tian}, we prove that the {\it transverse scalar curvature} and the diameter are both uniformly bounded along the Sasaki-Ricci flow. The frame work of the proof is very similar as in the K\"ahler setting \cite{Sesum-Tian}. But there is a major difference for $\cW$ functional on Sasakian manifolds and K\"ahler manifolds, namely, the $\cW$ functional on a Sasakian manifold is only involved with {\it basic functions}. While the distance function, for example, is not a basic function. To overcome this difficulty, we introduce a {\it transverse distance} on Sasakian manifolds, or more precisely, on {\it quasi-regular} Sasakian manifolds and we will see that  the Sasakian structure plays a central role. Nevertheless, this allows us to prove the results  along the Sasaki-Ricci flow  on quasi-regular Sasakian manifolds, such as the uniform bound of transverse scalar curvature and the diameter along the flow. For {\it irregular Sasakian structures}, one can  approximate the irregular structure by quasi-regular structures, using  Rukimbira's results \cite{Rukimbira95a}. Using the estimates in \cite{SWZ}, the corresponding Sasaki-Ricci flow for irregular structure can  be approximated by the Sasaki-Ricci flow for quasi-regular structure. One observation is that the transverse geometric quantities are uniformly bounded under the approximation, such as the transverse scalar curvature.   With the help of this approximation, one can show the transverse scalar curvature is  uniformly bounded along the Sasaki-Ricci flow for irregular structure. To show the diameter is uniformly bounded, we need the fact that on any compact Sasakian manifold (or more generally, on compact K-contact manifold), there always exists closed orbits of the Reeb vector field (\cite{Banyaga90}, see \cite{Rukimbira95, Rukimbira99} further development also).  Our results give some evidence  that the Sasaki-Ricci flow would converge in a suitable sense to a Sasaki-Ricci solition \cite{SWZ}.  As a direct consequence, we can prove that the Sasaki-Ricci flow exists for all positive time and converges to a {\it Sasaki-Ricci soliton} when the dimension is three. This result  was announced in \cite{WZ}. 

The Frankel conjecture states that a compact K\"ahler manifold of complex dimension $n$ with positive bisectional curvature is biholomorphic to the complex projective space $\mathbb{CP}^n$. The Frankel conjecture was proved  by  Siu-Yau \cite{Siu-Yau} using harmonic maps and Mori \cite{Mori} via algebraic methods. There have been extensive study of K\"ahler manifolds with positive (or nonnegative) bisectional curvature  using the K\"ahler-Ricci flow, for example to mention \cite{Bando, Mok, Chen-Tian, CCZ, PSSW} to name a few.  Since the positivity of bisectional curvature is preserved \cite{Bando, Mok},   the K\"ahler-Ricci flow will converge to a K\"ahler-Ricci soliton if the initial metric has nonnegative bisectional curvature \cite{CCZ}.  Recently Chen-Sun-Tian \cite{CST} proved that a compact K\"ahler-Ricci soliton with positive bisectional curvature is biholomorphic to a complex projective space  without using the resolution of the Frankel conjecture. Hence their result together with the previous results gives a proof of the Frankel conjecture via the K\"ahler-Ricci flow. 

It is then very natural to study the Sasakian manifolds with positive curvature in suitable sense. 
We use the Sasaki-Ricci flow to  study Sasakian manifolds  when the {\it transverse holomorphic bisectional curvature} is positive. The transverse (holomorphic) bisectional curvature is defined to be the bisectional curvature of the transverse K\"ahler structure.  Following the proof in the K\"ahler setting \cite{Bando, Mok}, one can show that the positivity of transverse bisectional curvature is preserved along the Sasaki-Ricci flow. It follows that  the transverse bisectional curvature is bounded by its transverse scalar curvature, hence it is bounded. It then follows that the Sasaki-Ricci flow will converge to a Sasaki-Ricci soliton in suitable sense. It would then be very interesting to classify  Sasaki-Ricci solitons with positive (or  nonnegative) transverse bisectional curvature. The success of the classification of Sasaki-Ricci solitons with positive (or  nonnegative) transverse bisectional curvature would lead to classification  of Sasakian manifolds with positive (nonnegative) transverse bisectional curvature, parallel to the results in \cite{Siu-Yau, Mori, Bando, Mok}. 

The organization of the paper is as follows: in Section 2 and Section 3, we summarize definition and some facts of Sasakian manifolds, the transverse K\"ahler structure and the Sasaki-Ricci flow. In Section 4 we introduce Perelman's $\cW$ functional on Sasakian manifolds and prove that it is monotone along the Sasaki-Ricci flow. In Section 5 we prove that the Ricci potential and the transverse scalar curvature are bounded in terms of the diameter along the flow. This section is pretty much the same as in the K\"ahler setting \cite{Sesum-Tian}.  In Section 6 we prove that the diameter is uniformly bounded along the Sasaki-Ricci flow if the Sasakian structure is regular or quasi-regular. In Section 7 we use the approximation mentioned above to prove that the diameter is uniformly bounded along the flow if the Sasakian structure is irregular. In Section 8 we study the Sasaki-Ricci flow for the initial metric with positive (nonnegative) transverse bisectional curvature; in particular we prove that the Sasaki-Ricci flow converges to a Sasaki-Ricci soliton with such an initial metric. In appendix we summarize some geometric and topological results of Sasakian manifolds with positive transverse bisectional curvature.

\begin{rmk}
On Mar 29th, 2011, about one day before this article posted on arxiv.org, I found T. Collins posted his paper  arXiv:1103.5720 in which he proved that Perelman's results on Kahler-Ricci flow  can be generalized to Sasaki-Ricci flow. I saw his paper before I posted this article which has a substantial overlap on generalization of Perelman's results to Sasaki-Ricci flow. 
\end{rmk}

\vspace{2mm}

{\bf Acknowledgement:} The author is grateful for Prof. Xiuxiong Chen and Prof. Jingyi Chen for constant support. The author is partially supported by an NSF grant. 

\section{Sasakian Manifolds}
In this section we recall definition and some basic facts  of Sasakian manifolds. For instance,  the recent monograph \cite{BG} is a nice reference for the details.  
Let $(M, g)$ be a $2n+1$ dimensional smooth Riemannian manifold ($M$ is assumed to be oriented, connected and compact unless specified otherwise), $\nabla$ the Levi-Civita connection of the Riemannian metric $g$, and let $R(X, Y)$ and $Ric$ denote the Riemannian curvature tensor and the Ricci tensor of $g$ respectively.   A Riemannian manifold $(M, g)$ is said to be a Sasakian manifold if and only if  the metric cone $\left(C(M)= M\times \R^{+},  \bar g=dr^2+r^2 g\right)$ is K\"ahler. $M$ is often identified with the submanifold $M\times \{1\}\subset C(M)$ $(r=1)$. Let $\bar J$ denote the complex structure of $C(M)$ which is compatible with $\bar g$. 
A Sasakian manifold $(M, g)$ inherits a number of geometric structures from the K\"ahler structure of its metric  cone. In particular, the {\it Reeb vector field} $\xi$ plays a very important role. The vector filed $\xi$ is defined as $\xi=\bar J(r\p_r)$. This  gives a 1-form $\eta(\cdot)=r^{-2}\bar g(\xi, \cdot)$. We shall  use $(\xi, \eta)$ to denote the  corresponding vector filed and 1-form on $M\cong M\times \{1\}$. One can see that
\begin{itemize}
\item $\xi$ is a Killing vector field on $M$ and $L_\xi\bar J=0$;
\item $\eta(\xi)=1, \iota_{\xi} d\eta(\cdot)=d\eta (\xi, \cdot)=0$;
\item the integral curves of $\xi$ are geodesics.  
\end{itemize}
The 1-form $\eta$ defines a vector sub-bundle $\cD$ of the tangent bundle $TM$ such that
 $\cD=ker(\eta)$.  There is an orthogonal decomposition of the tangent bundle \[TM=\cD\oplus L\xi,\] where $L\xi$ is the trivial  bundle generalized by $\xi$.
We can then introduce the $(1, 1)$ tensor field $\Phi$ such that
\[
\Phi(\xi)=0 \;\mbox{and}\; \Phi(X)=\bar JX, X\in \Gamma(\cD),
\]
where $M$ is identified with $M\times {1}\subset C(M)$. One can see that $\Phi$ can also be defined by
\[
\Phi(X)=\nabla_X \xi, X\in \Gamma(TM). 
\]
One can check that $\Phi$ satisfies \[\Phi^2=-I+\eta\otimes \xi \;\mbox {and}\; g(\Phi X, \Phi Y)=g(X, Y)-\eta(X)\eta(Y).\] 
 Moreover $\Phi$ is compatible with the 2-form $d\eta$ 
\[
d\eta(\Phi X, \Phi Y)=d\eta(X, Y),  X, Y\in \Gamma(TM).
\]
One can also check that
\[
g(X, Y)=\frac{1}{2} d\eta(X, \Phi Y), X, Y\in \Gamma(\cD). 
\]
Hence $d\eta$ defines a symplectic form on $\cD$, and $\eta$ is a contact form on $M$ since $\eta\wedge \left(\frac{1}{2}d\eta\right)^n$ is the volume form of $g$ and is then nowhere vanishing.  The following proposition shows some  equivalent descriptions of a Sasakian structure.  The proof can be found in \cite{BG} (Section 7).

\begin{prop}\label{P-2-1}
Let $(M, g)$ be a $2n+1$ dimensional Riemannian manifold. Then the following conditions are equivalent and can be used as the definition of a Sasakian structure.
\begin{enumerate}
\item $\left(C(M)\cong M\times \R^{+},  \tilde g=dr^2+r^2 g\right)$ is K\"ahler.
\item  There exists a Killing vector filed $\xi$ of unit length such that the tensor field $\Phi (X)=\nabla_X \xi$ satisfies
\begin{equation*}
(\nabla_X\Phi)(Y)=g(\xi, Y)X-g(X, Y)\xi.\end{equation*}

\item There exists a Killing vector filed $\xi$ of unit length such that the curvature satisfies
\begin{equation*}
R(X, \xi)Y=g(\xi, Y)X-g(X, Y)\xi.
\end{equation*}

\end{enumerate}
\end{prop}

Recall that the Reeb vector field $\xi$ defines a foliation $\cF_\xi$ of $M$ through its orbits. There is then a classification of Sasakian structures according to the global property of its orbits. If all the orbits are compact, hence circles, then $\xi$ generates a circle action on $M$. If the circle action is free, then $M$ is called {\it regular}. If the circle action is only locally free, then $M$ is called {\it quasiregular}. On the other hand, if $\xi$ has a non-compact orbit the Sasakian manifold is said to be {\it irregular}. We conclude this section by introducing the following structure theorem of a quasiregular Sasakian structure (see \cite{BG}, Theorem  7.1.3). 
\begin{theo}Let $(M, \xi, \eta, \Phi, g)$ be a quasi-regular Sasakian manifold with compact leaves. Let $Z=M/\cF_\xi$ denote the space of leaves of the Reeb foliation $\cF_\xi$.
Then

(1). The leaf space $Z$ is a compact complex orbifold with a K\"ahler metric $h$ and the K\"ahler form $\omega$,  which defines an integral class $[\omega]$ in $H^2_{orb}(Z, \Z)$. The canonical projection   $\pi: (M, g)\rightarrow (Z, h)$ is an orbifold Riemannian submersion. The fibers of $\pi$ are totally geodesic submanifolds of $M$ diffeomorphic to $S^1$. 

(2). $M$ is the total space of a principle $S^1$-orbibundle over $Z$ with connection $1$-form $\eta$ whose curvature $d\eta=\pi^{*}\omega$.  

(3). If the foliation $\cF_\xi$ is regular, then $M$ is the total space of a principle $S^1$-bundle over the  K\"ahler manifold $(Z, h)$. 
\end{theo}

\section{Transverse K\"ahler Structure and the Sasaki-Ricci Flow}
In this section we recall the {\it transverse K\"ahler structure} of Sasakian manifolds, which is a special case of transverse Riemannian structure. We shall introduce the transverse K\"ahler structure both globally (coordinate free) and locally (in local coordinates). These two description are equivalent and readers are referred to \cite{BG}, Section 2.5 for some more details about  transverse Riemannian structure. 

Let $M$ be a Sasakian manifold with $(\xi, \eta, \Phi, g)$.  Recall that
$\cD=ker(\eta)$ is a $2n$ dimensional subbunble of $TM$. 
If we denote
\[
J:=\Phi|_\cD,
\]
then $J$ is a complex structure on $\cD$ and it is compatible with $d\eta$. 
Hence $(\cD, J, d\eta)$ defines a K\"ahler metric on $\cD$. We define the  transverse K\"ahler metric $g^T$ as
\begin{equation}\label{E-1}
g^T(X, Y)=\frac{1}{2}d\eta (X, \Phi Y), X, Y\in \Gamma(\cD). 
\end{equation}
The metric $g^T$ is related to the Sasakian metric $g$ by
\begin{equation}\label{E-2}
g=g^T+\eta\otimes \eta. 
\end{equation}
 From the transverse metric $g^T$, we can define a connection on $\cD$ by
 \begin{equation}\label{E-3}
 \begin{split}
 \nabla^T_XY&=\left(\nabla_XY\right)^p, X, Y\in \Gamma(\cD)\\
 \nabla^T_\xi Y&=[\xi, Y]^p, Y\in \Gamma(\cD),
 \end{split}
 \end{equation}
where $X^p$ denotes the projection of $X$ onto $\cD$. One can check that this connection is the unique torsion free such that $\nabla^T g^T=0$.  The connection $\nabla^T$ is called the {\it transverse Levi-Civita connection}. 
Note that by Proposition \ref{P-2-1} (2), one can get that
\[
\nabla^T J=0. 
\]

We can further define the {\it transverse curvature operator}
by
\begin{equation}
R^T(X, Y)Z=\nabla^T_X\nabla^T_YZ-\nabla^T_Y\nabla^T_XZ-\nabla^T_{[X, Y]}Z.
\end{equation}
One can easily check that
\[
R^T(X, \xi)Y=0.
\]
Also when $X, Y, Z, W \in \Gamma(\cD)$, we have the following relation
\begin{equation}\label{E-2-curvature}
R(X, Y, Z, W)=R^T(X, Y, Z, W)-g(\Phi Y, W)g(\Phi X, Z)+g(\Phi X, W) g(\Phi Y, Z).
\end{equation}

The {\it transverse Ricci curvature} is then defined by
\[
Ric^T(X, Y)=\sum_i g\left(R^T(X, e_i)e_i, Y\right)=\sum_i g^T\left(R^T(X, e_i)e_i, Y\right),
\]
where $\{e_i\}$ is an orthonormal basic of $\cD$.  When $X, Y\in \Gamma(D)$, one can get that
\begin{equation}\label{E-5}
Ric^T(X, Y)=Ric(X, Y)+2g(X, Y)=Ric(X, Y)+2g^T(X, Y). 
\end{equation}
Hence for the {\it transverse scalar curvature} one can get that
\begin{equation}\label{E-4}
R^T=R+2n.
\end{equation}

\begin{defi}A Sasakian manifold is $\eta$-Einstein if there are two constants $\lambda$ and $\mu$ such that
\[
Ric=\l g+\mu \eta\otimes \eta. 
\]
\end{defi}
By Proposition 2.1, we have   $Ric(\xi, \xi)=2m$, hence $\l+\mu=2n$. 

\begin{defi}A Sasakian manifold is transverse K\"ahler-Einstein if there is a constant $\lambda$ such that
\[
Ric^T=\lambda g^T. 
\]
\end{defi}
By \eqref{E-5}, a Sasakian manifold is $\eta$-Einstein if and only if it is a transverse K\"ahler-Einstein. 
\begin{defi}
A Sasakian manifold $(M, g)$ is Sasakian-Einstein if $Ric=2n g$. 
\end{defi}
Note that a Sasakian manifold which is Sasakian-Einstein is necessary Ricci positive.

We shall also recall  the Sasakian structure and its transverse structure on local coordinates. For the details, see \cite{FOW}  (Section 3) for example. Let $(M, \xi, \eta, \Phi, g)$ be the Sasakian metric and  let  $g^T$ be the transverse K\"ahler metric.  
Let $U_\alpha$ be an open covering of $M$ and $\pi_\alpha: U_\alpha\rightarrow V_\alpha\subset \C^n$ submersions
such that 
\[
\pi_\alpha\circ \pi^{-1}_\beta: \pi_\beta(U_\alpha\cap U_\beta)\rightarrow \pi_\alpha (U_\alpha\cap U_\beta)
\]
is biholomorphic when $U_\alpha\cap U_\beta$ is not empty. One can choose  local coordinate charts $(z_1, \cdots, z_n)$ on $V_\alpha$ and local coordinate charts $(x, z_1, \cdots, z_n)$ on  $U_\alpha\subset M$  such that $\xi=\p_x$, where we use the notations
\[
\p_x=\frac{\p}{\p x}, \p_i=\frac{\p}{\p z_i}, \bar \p_{ j}=\p_{\bar j}=\frac{\p}{\p \bar z_{ j}}=\frac{\p}{\p z_{\bar j}}. 
\]
The map $\pi_\alpha: (x, z_1, \cdots, z_n)\rightarrow (z_1, \cdots, z_n)$ is then the natural projection. There is an isomorphism, for any $p\in U_\alpha$,
\[
d\pi_\alpha: \cD_p\rightarrow T_{\pi_\alpha(p)}V_\alpha. 
\]
Hence the restriction of $g$ on $\cD$ gives a well defined Hermitian metric $g^T_\alpha$ on $V_\alpha$ since $\xi$ generates isometries of $g$. 
One can actually verify that there is a well defined K\"ahler metric $g_\alpha^T$ on each $V_\alpha$ and 
\[
\pi_\alpha\circ \pi^{-1}_\beta: \pi_\beta(U_\alpha\cap U_\beta)\rightarrow \pi_\alpha (U_\alpha\cap U_\beta)
\]
gives an isometry of K\"ahler manifolds $(V_\alpha, g^T_\alpha)$.  The collection of K\"ahler metrics $\{g^T_\alpha\}$ on $\{V_\alpha\}$ can be used as an alternative definition of the transverse K\"ahler metric.
This definition is essentially  just the description of a transverse Riemannian geometry of a Riemannian foliation in terms of {\it Haefliger cocycles}, see \cite{BG} Section 2.5
 for more details.  We shall also use $\nabla^T_\alpha$, $R^T_\alpha(X, Y), Ric^T_\alpha$ and $R^T_\alpha$ for its Levi-Civita connection, the curvature, the Ricci curvature and the scalar curvature. 
The two definition of the transverse K\"ahler structure are actually equivalent.  We can see that 
$(\cD\otimes \C)^{(1, 0)}$ is spanned by the vectors of the form $\{\p_i+a_i\p_x\}, 1\leq i\leq n$, where $a_i=-\eta(\p_i)$. It is clear that
\[
d\eta(\p_i+a_i\p_x, \overline{\p_j+a_j\p_x})=d\eta(\p_i, \p_{\bar j}), 
\]
Hence the K\"ahler form of $g^T_\alpha$ on $V_\alpha$ is then the same as $\frac{1}{2}d\eta$ restricted on the slice $\{x=constant\}$ in $U_\alpha$. Moreover  for any $p\in U_\alpha$ and $X, Y\in \cD_p$,  
we have
\[
g^T(X, Y)=g^T_\alpha(d\pi_{\alpha} (X), d\pi_{\alpha}(Y)). 
\]
Hence $d\pi_\alpha: \cD_p\rightarrow T_{\pi_\alpha(p)}V_\alpha$ gives the isometry of $g^T$ on $\cD_p$ and $g^T_\alpha$ on $T_{\pi_\alpha(p)}V_\alpha$ for any $p\in U_\alpha$.  
For example one can easily check the following. For $X, Y, Z\in \Gamma(\cD)$ and $X_\alpha=d\pi_\alpha(X)\in TV_\alpha$, 
\[
\begin{split}
&d\pi_\alpha \left(\nabla^T_XY\right)=\left(\nabla^T_\alpha\right)_{X_\alpha} Y_\alpha =d\pi_\alpha (\nabla_XY), \\
&d\pi_\alpha\left(R^T(X, Y)Z\right)=R^T_\alpha(X_\alpha, Y_\alpha)Z_\alpha. 
\end{split}
\]
In this paper we shall treat the transverse K\"ahler metric by these two equivalent descriptions. For example, 
one can then easily verify curvature identities for $R^T$ since $R^T_\alpha$ is actually the curvature of  the K\"ahler metric $g^T_\alpha$ on $V_\alpha$. And it is very convenient to  do local computations using  $g^T_\alpha$ on $V_\alpha$.  While when we deal with some global features of $g^T$, such as integration  by parts, we shall use $g^T$. For simplicity, we shall suppress the index $\alpha$ without emphasis even when we do local computations on $V_\alpha$.

\begin{defi} A $p$-form $\theta$ on $M$ is called basic if
\[
\iota_\xi \theta=0, L_\xi \theta=0.
\]
Let $\Lambda^p_B$ be the sheaf of germs of basic $p$-forms and $\Omega^p_B=\Gamma(S, \Lambda^p_B)$ the set of sections of $\Lambda^p_B$.  
\end{defi}
The exterior differential preserves basic forms. We set $d_B=d|_{\Omega^p_B}$. 
Thus the subalgebra $\Omega_{B}(\cF_\xi)$ forms a subcomplex of the de Rham complex, and its cohomology ring $H^{*}_{B}(\cF_\xi)$  is called the {\it basic cohomology ring}. In particular, there is a transverse Hodge theory \cite{EKAH86, KT87, Ton97}. The transverse Hodge star operator $*_{B}$ is defined in terms of the usual Hodge star by
\[
*_{B} \alpha=*(\eta\wedge \alpha). 
\]
The adjoint $d^{*}_B: \Omega^p_{B}\rightarrow \Omega^{p-1}$ of $d_B$ is defined by
\[
d^*_{B}=-*_{B} d_B *_{B}. 
\]
The {\it basic Laplacian} operator $\t_B=d_B d^{*}_B+d^{*}_Bd_B$. The space $\cH^p_B(\cF_\xi)$ of {\it basic harmonic $p$-forms } is then defined to be the kernel of $\t_B: \Omega^p_{B}\rightarrow\Omega^p_{B}$. The transverse Hodge Theorem \cite{EKAH86} then says that each basic cohomology class has a unique harmonic representative. 

When $(M, \xi, \eta, g)$ is an Sasakian structure, there is a natural splitting of $\Lambda^p_B\otimes \C$ such that
\[
\Lambda^p_B\otimes \C=\oplus \Lambda^{i, j}_B,
\]
where $\Lambda^{i, j}_B$ is the bundle of type $(i, j)$ basic forms. We thus have the well defined operators
\[
\begin{split}
\p_B: \Omega^{i, j}_B\rightarrow \Omega^{i+1, j}_B,\\
\bar\p_B: \Omega^{i, j}_B\rightarrow \Omega^{i, j+1}_B.
\end{split}
\]
Then we have $d_B=\p_B+\bar \p_B$. 
Set $d^c_B=\frac{1}{2}\i\left(\bar \p_B-\p_B\right).$ It is clear that
\[
d_Bd_B^c=\i\p_B\bar\p_B, d_B^2=(d_B^c)^2=0.
\]
We  have the adjoint operators
\[
\p_B^{*}=-*_{B}\circ \bar\p_B\circ  *_{B}: \Omega^{i, j}_B\rightarrow \Omega^{i-1, j}_B
\]
and
\[
\bar \p_B^{*}=-*_{B}\circ \p_B\circ *_{B}: \Omega^{i, j}_B\rightarrow \Omega^{i, j-1}_B
\]
of $\p_B$ and $\bar \p_B$ respectively. For simplicity we shall use $\p, \bar \p, \bar \p^{*}$ and $\p^{*}$ if there is not confusion. We can define
\[
\t^B_{\p}=\p\p^{*}+\p^{*}\p, \t^B_{\bar \p}=\bar \p \bar \p^{*}+\bar\p^{*}\bar \p. 
\]
We define the operator $L: \Omega^{i, j}_B\rightarrow \Omega^{i+1, j+1}_B$ by
\[
L\alpha=\alpha\wedge \frac{1}{2}d\eta,
\]
and its adjoint operator $\Lambda: \Omega^{i+1, j+1}_B\rightarrow \Omega^{i, j}_B$ by
\[
\Lambda=-*_{B}\circ L\circ *_{B}
\]
As usual, one has
\begin{prop}
On a Sasakian manifold, one has
\begin{subequations}
\begin{align}
&[\Lambda, \p]=-\i \bar \p^{*},\label{E-kaha}\\
&[\Lambda, \bar\p]=\i \p^{*},\label{E-kahb}\\
&\p\bar \p^{*}+\bar\p^{*}\p=\p^{*}\bar\p+\bar\p \p^{*}=0,\\
&\t_{B}=2\t^{B}_{\p}=2\t^B_{\bar \p}. 
\end{align}\end{subequations}
\end{prop}
We shall also define the transverse Laplacian $\t^T$ for basic functions and basic forms as
\[
\t^T= \frac{1}{2}g^{i\bar j}_{T}\left(\nabla_i^T\nabla_{\bar j}^T+\nabla_{\bar j}^T\nabla_i^T\right)
\]
The operator $\t^T$ on basic forms is  well-defined. 
This can be seen easily from the description of the transverse K\"ahler structure in terms of Haefliger cocycles. For basic functions, we have
 \[\t^T f=g^{i\bar j}_T \nabla^T_{i}\nabla^T_{\bar j} f=g^{i\bar j}_T \frac{\p^2 f}{\p z_i\p z_{\bar j}}=-\t_{\bar \p }f. \] 

Now we define a 2-form $\rho^T$ called the transverse Ricci form as follows.
\[
\rho_\alpha^T=-\i\p\bar \p \log \det(g^T_\alpha).
\]

One can see that the pull back forms $\pi^*_\alpha \rho^T_\alpha$ patch together to give a global basic 2-form on $M$, which is denoted by $\rho^T$. As in the K\"ahler case $\rho^T$ is $d_B$ closed and define a basic cohomology class of type $(1, 1)$. The basic cohomology class $[\rho^T]$is independent of the choice of transverse K\"ahler form in the space of Sasaki space. The basic cohomology class $[\rho^T]/2\pi$ is called the basic first Chern class of $M$, and is denoted by $c_1^B(M)$. If there exists a transverse K\"ahler Einstein with
$Ric^T=\lambda \rho^T$ for some constant $\lambda$, then $c_1^B=\lambda [d \eta]_B$. This  implies that $c^B_1$ definite depending on the sign of $\lambda$ and $c_1(D)=0$ (see \cite{FOW, BG} for example).
 As was pointed out by Boyer, Galicki and Matzeu \cite{BGM}, in the negative and zero basic first class case with $c_1(D)=0$ the results of El Kacimi-Alaoui \cite{EKA} together with Yau's estimate \cite{Yau} imply that the existence of transverse K\"ahler-Einstein metric. The remaining case is $c_1^B=\lambda[d\eta]_B, \lambda>0$ and this is the case which is related closely to the Sasaki-Einstein metrics. In this case, Futaki-Ono-Wang \cite{FOW} proved that there is a Sasaki-Ricci soliton in $c_1^B$ if $M$ is toric in addition;  by varying the Reeb vector field, they can prove the existence of a Sasaki-Einstein metric on $M$. 
  
It is natural to deform one Sasakian structure to another, for example, to find Sasaki-Einstein metrics. 
There are many natural ways to deform a Sasakian structure, see \cite{BGS1} for example. In the present paper, we 
focus on the Sasaki-Ricci flow introduced in \cite{SWZ},
\begin{equation}
\frac{\p}{\p t}g^T=g^T-\lambda Ric^T,
\end{equation}
where $\lambda$ is a positive constant if $c_1^B>0$. By the so-called D-homothetic deformation, which is essentially just a rescaling of the transverse K\"ahler structure such that the new metric is still Sasakian, 
\[
\tilde \eta= a\eta, \tilde \xi=a^{-1} \xi, \tilde \Phi=\Phi, \tilde g=ag+a(a-1)\eta\otimes \eta,
\] we can assume that $\lambda=1$. 
Roughly speaking, the Sasaki-Ricci flow is to deform a Sasakian metric in such a way that the transverse K\"ahler metric is deformed along its Ricci curvature. In particular, it induces a K\"ahler-Ricci flow 
\[
\frac{\p}{\p t}g^T_\alpha=g^T_\alpha-Ric^T_\alpha
\]
on each $V_\alpha$. 
To study the Sasaki-Ricci flow, we are also  interested in the space of Sasakian metrics
\[
\cH=\{\phi\in C^\infty_B(M), \eta_\phi\wedge (d\eta_\phi)^n\neq 0, \eta_\phi=\eta+d^c_B\phi
\}.\]
This space is studied by Guan-Zhang \cite{GZ1, GZ2} and it  
can be viewed as an analogue of the space of K\"ahler metrics in a fixed K\"ahler class  \cite{Mabuchi87, Semmes96, Donaldson99}. 
For any $\phi\in \cH$, one can define a new Sasaki metric $(\eta_\phi, \xi, K_\phi, g_\phi)$ with the same Reeb vector field $\xi$ such that
\[
\eta_\phi=\eta+d^c_B\phi, K_\phi=K-\xi\otimes d^c_B\phi \circ K.
\]
In terms of the potential $\phi$, one can write the Sasaki-Ricci flow \cite{SWZ} as 
\begin{equation}\label{E-p-1}
\frac{\p \phi}{\p t}=\log\frac{\det(g^T_{i\bar j}+\phi_{i\bar j})}{\det(g_{i\bar j}^T)}+\phi-F,
\end{equation}
where $F$ satisfies
\[
\rho^T_g-d\eta=d_Bd_B^cF. 
\]

\section{Perelman's $\cW$ Functional on Sasakian Manifolds}

In this section we show that Perelman's $\cW$ functional \cite{Perelman01} has its counterpart on Sasakian manifolds. 
First we go over Perelman's $\cW$ functional on K\"ahler manifolds. In particular we repeat some computations of $\cW$ functional  when the metric $g$ is only allowed to vary  as K\"ahler metrics, which would be used when we consider $\cW$ functional on Sasakian manifolds.  On Sasakian manifolds,  a natural requirement is  that the functions appear in $\cW$ functional are basic.  This requirement fits the Sasakian structure very well and $\cW$ functional behaves well under the Sasaki-Ricci flow.  More or less, it can be viewed as  Perelman's $\cW$ functional for the transverse K\"ahler structure. Certainly  there are some different features in the Sasakian setting. For example, the Sasakian structure is  not preserved under scaling and the $\cW$ functional on Sasakian manifolds is  not scaling invariant either. 

First we recall the Perelman's $\cW$ functional when $(M, g)$ is assumed to be a  compact K\"ahler manifold of complex dimension $n$. 
Recall the $\cW$ functional, for $\tau>0$, 
\begin{equation}\label{E-3-1}
\cW(g, f, \tau)=\int_M\left(\tau(R+|\nabla f|^2)+f\right)e^{-f}\tau^{-n}dV,
\end{equation}
where $f\in C^\infty(M)$ satisfies the constraint
\[
\int_M e^{-f}\tau^{-n}dV=1.
\]
Let $z=(z_1, z_2, \cdots, z_n)$ be a holomorphic coordinate chart on $M$. 
We use the notation  \[
 |\nabla f|^2=g^{i\bar j} f_if_{\bar j}=f_if_{\bar i},
 \]
where we  use $g_{i\bar j}$ to denote the metric $g$  and $g^{i\bar j}$ to denote its inverse.

Suppose $\delta \tau=\sigma$, $\delta f=h$ , $\delta g_{i\bar j}=v_{i\bar j}$ and $v=g^{i\bar j} v_{i\bar j}$. We also have $\delta dV=v dV$. Note that we assume there is for some function $\psi$,  at least locally, such that
$v_{i\bar j}=\p_i\p_{\bar j}\psi$. Hence we will have $v_{i\bar j, j}=v_{, i}$. 

\begin{prop}\label{P-3-1} We have the following first variation of $\cW$, 
\begin{equation}\label{E-3-2}
\begin{split}
\delta W(g, f, \tau)=&\int_M \left(\sigma(R+\t f)-\tau v_{j\bar i}(R_{i\bar j}+f_{i\bar j})+h\right)\tau^{-n}e^{-f}dV\\
&+\int_M \left(2\t f-f_if_{\bar i}+R\right)\left(v-h-n\sigma \tau^{-1}\right)\tau^{-n}e^{-f}dV.
\end{split}
\end{equation}
\end{prop}
\begin{proof}  Since we shall use this computation in the Sasakian setting, we include the computations as follows. 
We compute
\[
\begin{split}
&\delta R= -v_{j\bar i} R_{i\bar j}-\t v,\\
& \delta (g^{i\bar j}f_if_{\bar j})=-v_{j\bar i}f_if_{\bar j}+f_i h_{\bar i}+h_i f_{\bar i},\\
& \delta (\tau^{-n}e^{-f}dV)=(v-h-n\sigma \tau^{-1})\tau^{-n}e^{-f} dV,
 \end{split}
\]
where we use the following notation for tensors
\[
\langle v_{i\bar j}, R_{i\bar j}\rangle=v_{i\bar j} R_{k\bar l}g^{i\bar l}g^{k\bar j}=v_{\bar i j} R_{i\bar j}. 
\]
Hence we compute
\begin{equation}\label{E-3-3}
\begin{split}
\delta \cW=&\int_M\left (\sigma (R+f_if_{\bar i})+h\right)\tau^{-n}e^{-f}dV\\
&+\int_M\tau\left(- v_{j\bar i}(R_{i\bar j}+f_if_{\bar j})-\t v+f_i h_{\bar i}+h_i f_{\bar i}\right) \tau^{-n}e^{-f}dV\\
&+\int_M \left(\tau(R+f_if_{\bar i})+f\right)(v-h-n\sigma\tau^{-1})\tau^{-n}e^{-f}dV. 
\end{split} 
\end{equation}

We can compute, integration by parts, 
 \begin{subequations}\label{E-3-4}
\begin{align}
&\int_M \t v e^{-f}dV=\int_M v(\t e^{-f})dV=\int_M v(-\t f+f_if_{\bar i})e^{-f}dV,\label{E-3-4a}\\
&\int_M (f_ih_{\bar i}+h_if_{\bar i})e^{-f}dV=-2\int_M h(\t f-f_if_{\bar i})e^{-f}dV,\label{E-3-4b}\\
&\int_M v_{j\bar i}f_{i\bar j}e^{-f}dV= \int_Mv_{j\bar i}f_if_{\bar j}e^{-f}dV+\int_M v(\t f-f_if_{\bar i})e^{-f}dV.\label{E-3-4c}
\end{align}
\end{subequations}

Applying \eqref{E-3-4} to \eqref{E-3-3}, then straightforward computations show that
\[
\begin{split}
\delta W(g, f, \tau)=&\int_M \left(\sigma(R+\t f)-\tau v_{j\bar i}(R_{i\bar j}+f_{i\bar j})+h\right)\tau^{-n}e^{-f}dV\\
&+\int_M \left(\tau(2\t f-f_if_{\bar i}+R)+f\right)\left(v-h-n\sigma \tau^{-1}\right)\tau^{-n}e^{-f}dV,
\end{split}
\]
where we have uses the fact that
\[\int_M e^{-f}(\t f-f_{i}f_{\bar i})dV=0.\]
\end{proof}

If we choose
\[
v_{i\bar j}=\l g_{i\bar j}^T-(R_{i\bar j}+f_{i\bar j}), h=v-n\sigma \tau^{-1}=\l n-n\sigma \tau^{-1}-(R+\t f),
\]
where $\l$ is a constant in \eqref{E-3-2}. 
It then follows that
\[
\begin{split}
\delta \cW(g, f, \tau)&=\int_M\left( \tau |R_{i\bar j}+f_{i\bar j}|^2-(\l \tau -\sigma+1)(R+\t f)+\l n-n\sigma \tau^{-1}\right)\frac{e^{-f}}{\tau^n} dV\\
&=\int_M  \left|R_{i\bar j}+f_{i\bar j}-\frac{\l \tau-\sigma+1}{2\tau}g_{i\bar j}\right|^2\tau^{-n+1}e^{-f}dV-\frac{n}{4\tau}(\l \tau-\sigma-1)^2.
\end{split}
\]
We choose $\sigma=\l \tau-1$, then we get
\begin{equation}\label{E-3-5}
\delta \cW(g, f, \tau)=\int_M \left|R_{i\bar j}+f_{i\bar j}-\tau^{-1}g_{i\bar j}\right|^2 \tau^{-n+1} e^{-f}dV. 
\end{equation}

We then consider the following  evolution equations, which is the coupled K\"ahler-Ricci flow,
\begin{equation}\label{E-3-8}
 \left\{
 \begin{array}{cl}
\dfrac{\p g_{i\bar j}}{\p t}&=g_{i\bar j}-R_{i\bar j},\\
\\
\dfrac{\p f}{\p t}&=n\tau^{-1}-R-\t f+f_{i}f_{\bar i},\\
\\
\dfrac{\p\tau}{\p t}&=\tau-1.
\end{array}
\right.
\end{equation}

\begin{prop}\label{P-3-2}Under the evolution equations \eqref{E-3-8}, we have
\begin{equation}\label{E-3-9}
\frac{d\cW}{dt}=\int_M \left|R_{i\bar j}+f_{i\bar j}-\tau^{-1}g_{i\bar j}\right|^2 \tau^{-n+1} e^{-f}dV+\int_M |f_{ij}|^2 \tau^{-n+1}e^{-f}dV. 
\end{equation}
\end{prop}
\begin{proof}
This is Perelman's  formula on K\"ahler manifolds. 
 We will include the computations here since the similar computations are needed for the Sasakian case. By \eqref{E-3-2} and \eqref{E-3-8}, we compute
 \begin{equation}\label{E-3-10}
 \begin{split}
 \frac{d\cW}{dt}=& \int_M \left((\tau-1)(R^T+\t f)+\tau(R_{i\bar j}-g_{i\bar j})(R_{\bar i j}+f_{\bar i j})\right)\tau^{-n}e^{-f}dV\\
 &+\int_M \left(n\tau^{-1}-R-\t f+f_if_{\bar i}\right)\tau^{-n}e^{-f}dV\\
 &+\int_M \left(\tau\left(2\t f-f_if_{\bar i}+R \right)+f\right)(\t^T f-f_if_{\bar i})\tau^{-n} e^{-f}dV\\
 =&\int_M  \left|R_{i\bar j}+f_{i\bar j}-\tau^{-1}g_{i\bar j}\right|^2\tau^{-n+1}e^{-f}dV\\
 &-\int_M f_{i\bar j} (R_{\bar i j}+f_{\bar i j})\tau^{-n+1} e^{-f}dV+\int_M f_if_{\bar i}\tau^{-n}e^{-f} dV\\
 &+\int_M \left(\tau\left(2\t f-f_if_{\bar i}+R\right)+f\right)(\t f-f_if_{\bar i})\tau^{-n} e^{-f}dV. 
 \end{split}
 \end{equation}
 We then observe that
 \begin{equation}\label{E-3-11}
 \begin{split}
 \int_M f (\t f-f_if_{\bar i}) e^{-f} dV&=-\int_M f\t (e^{-f})dV\\
 &=-\int_M \t f e^{-f} dV\\
 &=-\int_M f_i f_{\bar i}e^{-f} dV.
 \end{split}
 \end{equation}
 We  claim that
 \begin{equation}\label{E-3-12}
\begin{split}
 \int_M f_{i\bar j} (R_{\bar i j}+f_{\bar i j})e^{-f}dV=&\int_M \left(2\t f-f_if_{\bar i}+R\right)(\t  f-f_if_{\bar i}) e^{-f} dV\\
&-\int_M |f_{ij}|^2e^{-f} dV. 
\end{split}
 \end{equation}
 The proof is complete by \eqref{E-3-10}, \eqref{E-3-11} and \eqref{E-3-12}. 
 To prove \eqref{E-3-12}, we compute, integration by parts,
 \begin{equation}\label{E-3-13}
 \begin{split}
&\int_M f_{i\bar j} (R_{\bar i j}+f_{\bar i j})e^{-f}dV=\int_M \left(-f_i (R_{\bar i j}+f_{\bar i j})_{,\bar j} + f_i f_{\bar j}(R_{\bar i j}+f_{\bar i j})\right)e^{-f}dV\\
 &=\int_M\left(-f_i (R_{,\bar i} +(\t f)_{\bar i})+ f_i f_{\bar j}(R_{\bar i j}+f_{\bar i j})\right)e^{-f}dV\\
 &=\int_M\left(\t f-f_if_{\bar i})(R+\t f)+ f_i f_{\bar j} (R_{\bar i j}+f_{\bar i j})\right)e^{-f}dV. 
 \end{split}
 \end{equation}
 We can then compute that
 \begin{equation}\label{E-3-14}
 \int_M f_i f_{\bar j} f_{\bar i j}e^{-f}dV=\int_M \left(-f_{ij}f_{\bar j}f_{\bar i}- \t f|\nabla f|^2+ |\nabla f|^4\right)e^{-f}dV.
 \end{equation}
 We also have
 \begin{equation}\label{E-3-15}
 R_{\bar i j} f_i=f_{ij\bar i}-f_{i\bar i j}.
 \end{equation}
 Hence we get that
 \begin{equation}\label{E-3-16}
 \begin{split}
 \int_M f_if_{\bar j} R_{\bar i j}e^{-f}dV&=\int_M f_{\bar j}(f_{ij\bar i}-f_{i\bar i j})e^{-f}dV\\
 &=\int_M \left(-|f_{ij}|^2+ f_{ij}f_{\bar i}f_{\bar j}+(\t f)^2-\t f |\nabla f|^2\right)e^{-f}dV.
 \end{split}
 \end{equation}
 By \eqref{E-3-14} and \eqref{E-3-16}, we get hat
 \[
 \int_{M} f_i f_{\bar j} (R_{\bar i j}+f_{\bar i j})e^{-f}dV=\int_M \left((\t f-|\nabla f|^2)^2-|f_{ij}|^2\right)e^{-f}dV.
 \]
 This with \eqref{E-3-13} proves the claim. 
 \end{proof}

From \eqref{E-3-9}, it follows directly that  $\cW$ is strictly increasing along the evolution equations \eqref{E-3-8} unless
\[
R_{i\bar j}+f_{i\bar j}-\tau^{-1}g_{i\bar j}=0, f_{ij}=0,
\] 
which implies that $\nabla f$ is a (real) holomorphic vector field and $g$ is a K\"ahler-Ricci (gradient-shrinking) soliton. 

Now we consider that $(M, \xi, \eta, \Phi, g)$ is a $2n+1$ dimensional Sasakian manifold. Let
\begin{equation}\label{E-3-17}
\cW(g, f, \tau)=\int_M \left(\tau(R^T+|\nabla f|^2)+f\right)e^{-f} dV, 
\end{equation}
where $f\in C^\infty_B(M)$ (namely $df(\xi)$=0) and satisfies
\begin{equation}\label{E-3-scale}
\int_M \tau^{-n}e^{-f}dV=1. 
\end{equation}
For $f$ basic, $|\nabla f|^2=|\nabla^T f|^2=g^{i\bar j}_Tf_i f_{\bar j}$. 
Recall that $R^T$ is the transverse scalar curvature and satisfies
\[
R^T=R+2n. 
\]
So $\cW$ functional in \eqref{E-3-17} can be viewed as Perelman's  $\cW$ functional restricted on basic functions and transverse Kahler structure. 
Note that the normalization condition \eqref{E-3-scale} is not exactly the same as in Perelman's $\cW$ functional; in particular, $\cW$ functional is not scaling invariant and is not invariant under $D$-homothetic deformation either. 

When $f$ is basic, one can see that the integrand in \eqref{E-3-17} is only involved  with the transverse K\"ahler structure. Hence when we compute the first variation of $\cW$, all the local computations are just mimic the computations as in the K\"ahler case. Under the Sasaki-Ricci flow, the Reeb vector field and the transverse holomorphic structure are both invariant, and the metrics are bundle-like; we shall see that when one applies integration by parts, the expressions involved behave essentially the same as in the K\"ahler case. Hence we would get the similar monotonicity for $\cW$ functional along the Sasaki-Ricci flow. 

One can choose a local coordinate chart $(x, z_1, \cdots, z_n)$ on a small neighborhood $U$ such that  (see \cite{GKN} for instance)
\begin{equation}\label{A-3-1}
\begin{split}
\xi&=\p_x\\
 \eta&=dx-\i \left(G_j dz_j-G_{\bar j} d z_{\bar j}\right)\\
 \Phi&=\i \left(X_j\otimes dz_j-X_{\bar j}d z_{\bar j}\right)\\
 g&=\eta\otimes \eta+g^T_{i\bar j}dz_i\otimes dz_{\bar j}, 
 \end{split}
 \end{equation}
where $G: U\rightarrow \R$ is a (local) real basic function such that $\p_xG=0$, and we use the notations $G_i=\p_i G, G_{\bar j}=\p_{\bar j}G$, $X_j=\p_i+\i G_i\p_x, X_{\bar j}=\bar X_i$,  and $g^T_{i\bar j}=2G_{i\bar j}=2\p_i\p_{\bar j}G$. Then $\cD\otimes \C$ is spanned by $\{X_i, X_{\bar j}\}, 1\leq i, j\leq n$. It is clear that
\[
\Phi X_i=\i X_i, \Phi X_{\bar j}=-\i X_{\bar j}. 
\]
We can also compute that
\[
[X_i, X_j]=[X_{\bar i}, X_{\bar j}]=[\xi, X_i]=[\xi, X_{\bar j}]=0
\]
and
\[
[X_i, X_{\bar j}]=-2\i G_{i\bar j}\p_x. 
\]
The transverse K\"ahler metric is given by $g^T=g_{i\bar j}^Tdz_i\otimes dz_{\bar j}$, where $g^T_{i\bar j}=g^T(\p_i,  \p_j)=g^T(X_i, X_j)=2G_{i\bar j}$. We also use $g^{i\bar j}_T$ to denote the inverse of $g^T_{i\bar j}$.

We consider that the variation of the Sasakian structure which fixes $\xi$ and the transverse holomorphic structure such that $\delta g^T_{i\bar j} =v_{i\bar j}=v(X_i,  X_j)=v(\p_i,  \p_j)$, and there is a basic function $\psi$ such that $v_{i\bar j}=\psi_{i\bar j}$. Let $v=g_T^{i\bar j}v_{i\bar j}$. Hence we have $v_{i\bar j}, _{j}=v_{, i}$ ( this holds if  there exists a local basic function $\psi$ such that $v_{i\bar j}=\psi_{i\bar j}$). Suppose $\delta f=h\in C^\infty_B(M), \delta \tau=\sigma$, then 
\begin{prop}\label{P-3-3}The first variation of $\cW$ on Sasakian manifolds is given by
\begin{equation}\label{E-3-18}
\begin{split}
\delta W(g, f, \tau)=&\int_M \left(\sigma(R^T+\t^T f)-\tau v_{j\bar i}(R^T_{i\bar j}+f_{i\bar j})+h\right)\tau^{-n}e^{-f}dV\\
&+\int_M \left(2\t^T f-|\nabla^T f|^2+R^T\right)\left(v-h-n\sigma \tau^{-1}\right)\tau^{-n}e^{-f}dV.
\end{split}
\end{equation}\end{prop} 
\begin{proof}The proof is almost identical to Proposition \ref{P-3-1}. First of all, since the integrand involves only the transverse K\"ahler structure,  the local computations are exactly the same  in terms of transverse K\"ahler structure. For example, the variation of the transverse scalar curvature takes the form of 
\[
\begin{split}
\delta R^T&=\delta (g_T^{i\bar j}R^T_{i\bar j})= -\delta (g_T^{i\bar j}\p_i\p_{\bar j}\log \det (g_{i\bar j})^T)\\
&=-\t^T v-\langle v_{i\bar j}, R_{i\bar j}^T\rangle_{g^T}. 
\end{split}
\]
Hence we need to check that integration by parts as in \eqref{E-3-4} also holds with the similar formula in the Sasakian case. 
We shall give a proof of \eqref{E-3-4c}; while \eqref{E-3-4a} and \eqref{E-3-4b} are quite obvious.
To verify \eqref{E-3-4c}, we write the integrand as the inner product of two basic forms (see \eqref{E-9-i2} for the inner product of two basic forms),
\[
\int_M v_{i\bar j}f_{j\bar i}e^{-f}dV= (  v_{i\bar j}dz_{i}\wedge dz_{\bar j},  e^{-f}\p\bar \p f). 
\]
Let $\Phi=   v_{i\bar j}dz_{i}\wedge dz_{\bar j}$ and $\Psi= \p\bar \p h$ be two basic forms. 
We have
\begin{equation}
\ (\Phi, \Psi) =\i (\Lambda \Phi,  \t_{\bar \p} h). 
\end{equation}
Actually, using \eqref{E-kahb}, we can get 
\begin{equation}
\begin{split}
(\Lambda \Phi,  \t_{\bar \p} h) &=(\Lambda \Phi, \bar \p^{*}\bar\p h)\\
&= (\bar \p \Lambda \Phi, \bar \p h)\\
&=(\Lambda \bar \p\Phi- [\Lambda, \bar \p]\Phi, \bar \p h)\\
&= (-\i \p^{*}\Phi, \bar \p h)\\
&=-\i(\Phi,  \p\bar\p h)=-\i (\Phi, \Psi). 
\end{split}
\end{equation}

Note that $\p\bar \p (e^{-f})= -\p\bar \p f e^{-f}+ \p f \wedge \bar\p f e^{-f}$, we have
\begin{equation}
\begin{split}
\int_M v_{i\bar j}f_{j\bar i}e^{-f}&=\int_M (\Phi, - \p\bar \p (e^{-f}))dV+\int_M (\Phi,  \p f\wedge \bar \p f e^{-f})dV\\
&=\int_M (\i \Lambda \Phi, \t_{\bar \p} (e^{-f}))dV+\int_M v_{i\bar j} f_{\bar i}f_{j} e^{-f }dV\\
&=\int_M (\i \Lambda \Phi, -\t^T (e^{-f}))dV+\int_M v_{i\bar j} f_{\bar i}f_{j} e^{-f }dV\\
&=\int_M v (\t^T f-|\nabla^T f|^2)e^{-f} dV+\int_M v_{i\bar j} f_{\bar i}f_{j} e^{-f }dV,
\end{split}
\end{equation}
where we use the fact that $\i \Lambda \Phi=\Lambda (\i \Phi)=g^{i\bar j}_T v_{i\bar j}=v$ and $\t^T=-\t_{\bar \p}$ for basic functions. 
\end{proof}

We then consider the $\cW$ functional on the coupled Sasaki-Ricci flow
\begin{equation}\label{E-3-20}
 \left\{
 \begin{array}{cl}
\dfrac{\p g_{i\bar j}}{\p t}&=g^T_{i\bar j}-R^T_{i\bar j},\\
\\
\dfrac{\p f}{\p t}&=n\tau^{-1}-R^T-\t f+f_{i}f_{\bar i},\\
\\
\dfrac{\p\tau}{\p t}&=\tau-1.
\end{array}
\right.
\end{equation}
The evolution equation for $f$ in \eqref{E-3-20} is actually a {\it backward heat equation} where the quantities are computed by the metric $g$ at time $t$. Note that the Sasaki-Ricci flow always has a long time solution. For any time $T$, if  $f(T)\in C^\infty_B(M)$, we shall prove that the backward heat equation for $f$ in \eqref{E-3-20} has a unique smooth solution $f: [0, T]\rightarrow C^\infty_B(M)$ (see Proposition \ref{P-3-5} below).   Under the coupled Sasaki-Ricci flow \eqref{E-3-20}, we have
\begin{prop} \label{P-3-4}Under the coupled Sasaki-Ricci flow \eqref{E-3-20}, then $\cW$ functional on Sasakian manifolds satisfies
\begin{equation}\label{E-3-21}
\frac{d\cW}{dt}=\int_M \left|R^T_{i\bar j}+f_{i\bar j}-\tau^{-1}g^T_{i\bar j}\right|^2 \tau^{-n+1} e^{-f}dV+\int_M |f_{ij}|^2 \tau^{-n+1}e^{-f}dV. 
\end{equation}
\end{prop}

\begin{proof}Note that we assume that the function $f$ is basic. The computations then follow from Proposition \ref{P-3-3} and the proof as in Proposition \ref{P-3-2} once we show the corresponding formula of integration by parts as in \eqref{E-3-13}, \eqref{E-3-14} and \eqref{E-3-16} hold. All local computations are the same since the integrand involves only transverse K\"ahler structure.  We then finish our proof by showing the following general formula of integration by parts on Sasakian manifolds (see Proposition \ref{P-3-i} below). 
\end{proof}

Actually one can prove some general formula of integration by parts on Sasakian manifolds when the integrand is involved with basic functions and basic forms. The formula in Proposition \ref{P-3-i} implies that when the integrand involved only the transverse K\"ahler structure, basic functions and forms, integration by parts takes the formula as in the K\"ahler setting, with the corresponding terms involved with the K\"ahler metric replaced by the transverse K\"ahler metric. Hence Proposition \ref{P-3-3} and \ref{P-3-4} can be proved directly using the computations in K\"ahler setting (see Proposition \ref{P-3-1} and \ref{P-3-2}) and Proposition \ref{P-3-i} below. 

Let $\phi$ be a basic $(p, q)$ form ($\phi\in \Omega^{p, q}_B$) such that $\phi=\phi_{A\bar B}dz_{A}\wedge dz_{\bar B}, |A|=p, |B|=q$, where $A, B$  are  multi-index such that $A=\alpha_1\alpha_2\cdots \alpha_p$, $B=\beta_1\cdots \beta_q$ and $dz_A=dz_{\alpha_1}\wedge \cdots \wedge dz_{\alpha_p}$, $dz_{\bar B}= dz_{\bar \beta_1}\wedge\cdots \wedge dz_{\bar \beta_q}$. Then  the tensor
\[
g^{i\bar j}_T\nabla_i^T\nabla_{\bar j}^T \phi=g^{i\bar j}_T\nabla_i^T\nabla^T_{\bar j} \phi_{A\bar B} dz_{A}\wedge dz_{\bar B}
\]
is well defined and it is still a basic form. Again we can see this easily from the description of the transverse K\"ahler structure in terms of Haefliger cocycles. Then we have, 
\begin{prop}\label{P-3-i}Let $\phi$ and $\psi$ be two basic $(p, q)$ forms, then we have the following formula of integration by parts,
\begin{equation}\label{A-3-5}
\int_M g^{i\bar j}_T\nabla_i^T\nabla^T_{\bar j} \phi_{A\bar B} \overline{ \psi_{C\bar D}} g^{A\bar C}_T g^{B\bar D}_T dV=-\int_M g^{i\bar j}_T\nabla_{\bar j}^T\psi_{A\bar B} \overline{ \nabla_{\bar i}^T\psi_{C\bar D}} g^{A\bar C}_Tg^{B\bar D}_T dV. 
\end{equation}
\end{prop}
\begin{proof}For simplicity, we only prove that $\phi, \psi$ are $(1, 1)$ forms. The general case can be proved similarly. First denote the connection of the transverse K\"ahler metric as
\[\Gamma_{i j}^k=g^{k\bar l}_T\frac{\p g^T_{i\bar l}}{\p z_{j}}. \]
Choose a local coordinate chart $(x, z_1, z_2, \cdots, z_n)$ as in \eqref{A-3-1}.
Then we have 
\begin{equation}\label{A-3-2}
\begin{split}
g^{i\bar j}_T\nabla_i^T\nabla^T_{\bar j} \phi_{A\bar B}&= \nabla^T_{i} \left(g^{i\bar j}_T\nabla_{\bar j}^T \phi_{A\bar B}\right)\\
&=\frac{\p}{\p z_i}\left(g^{i\bar j}_T\nabla_{\bar j}^T \phi_{A\bar B}\right)+\Gamma^{i}_{ik}g^{k\bar j}_T\phi_{A\bar B}-\Gamma^{p}_{iA}g^{i\bar j}_T\nabla^T_{\bar j}\phi_{p\bar B},\\
\overline{\nabla_{\bar i}^T \psi_{C\bar D}} g^{A\bar C}_T g^{\bar B D}_T&= \nabla_i^T\left(\overline{\psi_{C\bar D}} g^{A\bar C}_T g^{\bar B D}_T\right) \\
&=\frac{\p}{\p z_i}\overline{\psi_{C\bar D}}g^{A\bar C}_T g^{\bar B D}_T +\Gamma^{A}_{ik}\overline{\psi_{C\bar D}} g^{k\bar C}_T g^{\bar B D}_T, \\
dV&=c\det(g_{i\bar j}^T) dx\wedge dZ\wedge d\bar Z, 
\end{split}
\end{equation}
where
\[
dZ=dz_1\wedge \cdots \wedge dz_n, d\bar Z=dz_{\bar 1}\wedge\cdots \wedge dz_{\bar n}, c=(\sqrt{-1})^n\frac{n!}{2^n}
\]
Now define a vector field
\[
X=g^{i\bar j}_T\nabla^T_{\bar j}\phi_{A\bar B}\overline{\psi_{C\bar D}} g^{A\bar C}_T g^{\bar B D}_T\frac{\p}{\p z_i}.
\]
One can check directly that $X$ is globally defined on $M$. Hence we can get a
$2n$ form $\alpha$ on $M$ such that
\[
\alpha=\iota_{X}dV=c (-1)^i g^{i\bar j}_T\nabla^T_{\bar j}\phi_{A\bar B}\overline{\psi_{C\bar D}} g^{A\bar C}_T g^{\bar B D}_T \det(g_{i\bar j}^T)dx\wedge dz_1\wedge \cdots \hat {dz_{i}} \cdots dz_{n}\wedge d\bar Z.
\]
Then straightforward computation shows that
\[
d\alpha=c\frac{\p}{\p z_i} \left(g^{i\bar j}_T\nabla^T_{\bar j}\phi_{A\bar B}\overline{\psi_{C\bar D}} g^{A\bar C}_T g^{\bar B D}_T \det(g_{i\bar j}^T)\right)dx\wedge dZ\wedge d\bar Z.
\]
Hence 
\[
\int_M \frac{\p}{\p z_i} \left(g^{i\bar j}_T\nabla^T_{\bar j}\phi_{A\bar B}\overline{\psi_{C\bar D}} g^{A\bar C}_T g^{\bar B D}_T \det(g_{i\bar j}^T)\right)dx\wedge dZ\wedge d\bar Z=0. 
\]
It follows that
\begin{equation}\label{A-3-3}
\begin{split}
&\int_M \frac{\p}{\p z_i} \left(g^{i\bar j}_T\nabla^T_{\bar j}\phi_{A\bar B}\right)\overline{\psi_{C\bar D}} g^{A\bar C}_T g^{\bar B D}_T \det(g_{i\bar j}^T)dx\wedge dZ\wedge d\bar Z\\&\;=
-\int_M
g^{i\bar j}_T\nabla^T_{\bar j}\phi_{A\bar B}\frac{\p}{\p z_i} \left(\overline{\psi_{C\bar D}} g^{A\bar C}_T g^{\bar B D}_T \det(g_{i\bar j}^T)\right)dx\wedge dZ\wedge d\bar Z\\
&=-\int_Mg^{i\bar j}_T\nabla^T_{\bar j}\phi_{A\bar B}\frac{\p}{\p z_i} \left(\overline{\psi_{C\bar D}} g^{A\bar C}_T g^{\bar B D}_T\right) \det(g_{i\bar j}^T)dx\wedge dZ\wedge d\bar Z\\
&\;-\int_Mg^{i\bar j}_T\nabla^T_{\bar j}\phi_{A\bar B} \overline{\psi_{C\bar D}} g^{A\bar C}_T g^{\bar B D}_T\frac{\p \det(g_{i\bar j}^T)}{\p z_i}dx\wedge dZ\wedge d\bar Z. 
\end{split}
\end{equation}
Note that
\begin{equation}\label{A-3-4}
\frac{\p}{\p z_k} \det(g_{i\bar j}^T)=\Gamma^{i}_{ik} \det(g_{i\bar j}^T). 
\end{equation}
Taking \eqref{A-3-2} and \eqref{A-3-4} into account, then \eqref{A-3-5} follows from \eqref{A-3-3}.
\end{proof}

Sometimes it is more convenient to write the $\cW$ functional as,  taking $w=e^{-f/2}$, 
\begin{equation}\label{E-3-22}
\cW(g, w, \tau)=\int_M\left(\tau \left(R^Tw^2+4|\nabla w|^2\right)-w^2\log w^2\right)\tau^{-n}dV,
\end{equation}
where $w\in C^\infty_B(M)$ such that $dw(\xi)=0$ and
\[
\int_M w^2\tau^{-n}dV=1. 
\]

When $w$ in $\eqref{E-3-22}$ is not assumed to be a basic function, Rothaus' result \cite{Rothaus} can be applied  to get that there is a nonnegative minimizer $w_0$. Furthermore, Rothaus \cite{Rothaus} showed that a nonnegative minimizer has to be positive everywhere and smooth. Hence one can get a smooth minimizer for \eqref{E-3-1} by letting $f_0=-2\log w_0$. When $w$ is assumed to be basic, Rothaus' results certainly imply that $\cW(g, w, \tau)$ is bounded from below, hence we can define $\mu$ functional 
as \begin{equation}\label{E-3-25}
\mu(g, \tau)=\inf_f \cW(g, f, \tau),
\end{equation} 
where $df(\xi)=0$ such that $\int_M \tau^{-n}e^{-f}dV=1$ and it is finite.
We could also mimic his proof (for \eqref{E-3-22}) by using  a variational approach  for $w\in W^{1, 2}_B(M)$ instead of $W^{1, 2}(M)$,  where $W^{1, 2}_B(M)$ is the closure of $C^\infty_B(M)$ in $W^{1, 2}(M)$. One can actually show that there is a nonnegative minimizer without too much change as in \cite{Rothaus}, Section 1.
We summarize Rothaus' results as follows.
\begin{theo}[Rothaus \cite{Rothaus}] \label{E-4-R}There is a  minimizer $w_0\in W^{1, 2}_B(M)$ of \eqref{E-3-22} which is nonnegative. Moreover, $w_0$ is strictly positive and smooth. 
\end{theo}
\begin{proof}
First of all, we show that there exists a nonnegative minimizer $w_0\in W^{1, 2}_B$. 
The statement is a special case considered by Rothaus applied to $W^{1, 2}_B$ instead of $W^{1, 2}$. The proof follows the line in \cite{Rothaus} and it  is almost  identical (\cite{Rothaus}, Section 1). The only difference is  to prove that a minimizer $w_0$ is actually  a weak solution of the equation,
\begin{equation}\label{E-w-1}
4\t w_0+w_0\log w_0^2-R^T w_0+\mu(g, \tau) w_0=0.
\end{equation}
The proof is slightly more involved and we need that $R^T$ is a basic function.  we shall prove this in Appendix. 
Once we show that  $w_0$ is a weak solution of \eqref{E-w-1} and it is bounded, we can actually apply
Rothaus' result  to show get that $w_0$ is strictly positive (see Lemma on page 114, \cite{Rothaus}); hence we would get a smooth minimizer. 
\end{proof}

Actually what we really need is the fact that the $\mu$ functional is nondecreasing along the Sasaki-Ricci flow, which would be enough for us in essence and does not really depend on the above theorem. 
 The proof is similar as in Ricci flow case if we replace a minimizer by a minimizing sequence.  
\begin{prop}\label{P-3-6}Let $\tau_0$ be a positive constant and 
\[
\tau(t)=(\tau_0-1)e^t+1>0, t\in [0, T]
\]
Then the functional $\mu(g(t), \tau(t))$ is monotone increasing along the Sasaki-Ricci flow on $[0, T]$. 
\end{prop}
\begin{proof}
It is clear that $\tau(t)$ satisfies $\p_t\tau=\tau-1$ with $\tau(0)=\tau_0$.
Suppose at time $T$, there is a minimizing sequence $\{f_k(T)\}$ such that
\[
\cW_k=\cW(g(T), f_k(T), \tau(T))\rightarrow \mu(g(T), \tau(T)), k\rightarrow \infty.
\]
By Proposition \ref{P-3-5} below, there is a basic function $f_k(t)$ satisfies \eqref{E-3-22}. Hence by Proposition \ref{P-3-4}, we have that
\[\cW_k(t)=\cW(g(t), f_k(t), \tau)\leq \cW_k.\] 
Hence $\mu(g(t), \tau(t))\leq \cW_k(t)\leq \cW_k$. Let $k\rightarrow \infty$, we get $\mu(g(t), \tau(t))\leq \mu(g(T), \tau(T))$.
\end{proof}

Now we consider the backward heat equation 
\begin{equation}\label{E-3-23}
\frac{\p f}{\p t}=n\tau^{-1}-R^T-\t f+|\nabla f|^2,
\end{equation}
where $\tau$ is a positive number depending only on time such that $\p_t \tau=\tau-1$. 
The corresponding backward heat equation for $w$ is
\begin{equation}\label{E-3-23a}
\frac{\p w}{\p t}=-\t w+(R^T-n\tau^{-1})w.
\end{equation}

\begin{prop}\label{P-3-5}Suppose  the Sasaki-Ricci flow exists in $[0, T]$, 
then for any $w(T)\in C^{\infty}(M)$, there is a unique solution $w(t)\in C^\infty(M)$ of \eqref{E-3-23a} for $t\in [0, T]$. If $w(T)\geq 0$, then $w(t)\geq 0$;  and if in addition $w(T)$ is not identically zero, then $w(t)>0$ for $t\in [0, T)$. Moreover if $w(T)\in C^\infty_B(M)$, then $w(t)\in C^{\infty}_B(M)$. As a consequence, 
then for any $f(T)\in C^\infty_B(M)$, there exists a unique solution $f(t)\in C^\infty_B(M)$ for any $t\in [0, T]$. If $w(T)\in W^{1, 2}_B$ is nonnegative, then $w(t)\in C^\infty_B$ and it is nonnegative. 
\end{prop}

\begin{proof}The proof is similar as in the Ricci flow case (c.f \cite{KL} for example) and $f$ is basic due to the maximum principle. 
Let $s=T-t$ and $\tilde f(s)=\exp(-f(s))$, then one can compute that
\[
\frac{\p \tilde f}{\p s}=\t_{g(s)}\tilde f-(R^T-n\tau^{-1})\tilde f,
\]
such that $\tilde f(0)\in C^\infty_B(M)$.   This is a linear equation for $\tilde f$ and there exists a unique smooth solution if $\tilde f (0)$ is smooth. 
Now if $\tilde f(0)\geq 0$, then $\tilde f\geq 0$ and if it is positive at one point in addition, we claim that $\tilde f>0$ for any $s\in [0, T]$. 
 Suppose $s_0\in (0, T]$ such that $s_0=T-t_0$. Let $h(t)$ satisfies the heat equation
 \[
 \frac{\p h}{\p t}=\t h
 \] 
on $(t_0, T]$ such that $\lim_{t\rightarrow t_0}h\rightarrow \delta_p$ for any $p\in M$, where $\delta_p$ is the delta function. Consider $\tilde f(T-t)$ for $t\in [t_0, T]$,
then by direct computations we can get that
\begin{equation}\label{E-3-24}
\frac{d}{d t} \int_M \tau^{-n}\tilde f(T-t) h(t)dV=0. 
\end{equation}
It is clear that $h>0$ in $(t_0, T]$. Then 
\[
\tilde f(p, T-t_0)=\int_M \tilde f(q, T-t)\delta_p(q)dV(q)=\tau^n(t_0)\lim_{t\rightarrow t_0} \int_M \tau^{-n} \tilde f(T-t) h(t)dV.
\]
Hence by \eqref{E-3-24} we get
\[
\tilde f(p, T-t_0)=\tau^n(t_0)\tau^{-n}(T)\int_M \tilde f(0) h(T)dV\geq 0.
\]
And it is clear that if $\tilde f(0)>0$ somewhere, then $\tilde f(p, T-t_0)>0$ for any $p$. 
Suppose $\tilde f(0)$ is basic, then for $\xi \tilde f$ we have
\[
\frac{\p (\xi \tilde f)}{\p s}=\t_{g(s)}(\xi \tilde f)-(R^T-n\tau^{-1})(\xi \tilde f).
\]
since $R^T$ is basic. Hence by the uniqueness of the solution we know that $\xi\tilde f\equiv0$ if $\xi\tilde f(0)=0$. Let $f(T-t)=-\log \tilde f(T-t)$, then it satisfies \eqref{E-3-23}. If $\tilde f(0)\in W^{1, 2}$, then by the standard regularity theory for the linear parabolic theory, we have $\tilde f(t)\in C^\infty$. 
\end{proof}

\section{The Ricci Potential and the Scalar Curvature along the  Flow}

We consider the (transverse) scalar curvature and the (transverse) Ricci potential along the Sasaki-Ricci flow.
The following discussions and computation follow closely the K\"ahler setting \cite{Sesum-Tian}.  

 First we show that there is a uniform lower bound on the transverse Ricci potential $u(x, t)$, which is defined by
 \begin{equation}\label{E-4-1}
 g^T_{i\bar j}-R^T_{i\bar j}=\p_i\bar \p_j u
 \end{equation}
with the normalized condition
\begin{equation}\label{E-4-2}
\int_M e^{-u}dV_g=1. 
\end{equation}
Note that $u\in C^\infty_B(M)$.   We compute
\[
\begin{split}
\p_t \p_i\bar \p_{ j} u(x, t)&= g_{i\bar j}^T-R^T_{i\bar j}+\p_t \p_i\bar \p_j \log (g^T_{i\bar j}(x, 0)+\p_i\bar\p_j \phi(x, t))\\
&=\p_i\bar \p_j (u+\t ^Tu).
\end{split}\]

Hence there exists a time dependent constant $A(t)$ such that
\begin{equation}\label{E-4-3}
\frac{\p u}{\p t}=\t^T u+u+A,\end{equation}
where 
\[
A=-\int_M ue^{-u} dV_g
\] in view of the normalized condition \eqref{E-4-2}. 

\begin{lemma} \label{L-4-1}$A(t)$ is uniformly bounded.
\end{lemma}

\begin{proof}
It is clear that $f(t)=te^{-t}\leq e^{-1}$ is bounded from above for any $t\in \R$. It then follows that
\[
A(t)=-\int_M ue^{-u}dV_g\geq -e^{-1}\int_M dV_g. 
\]
Note that the volume of $(M, \xi, g)$ is a constant  along the flow.

Let $\tau_0=1$ (hence $\tau\equiv 1$) and consider the functional $\mu(g(t), 1)$. Hence
\[
\mu(g(0), 1)\leq \mu(g(t), 1).
\]
Note that $R^T=n-\t^T u$; hence we get that
\[
\cW(g(t), u(t), 1)=\int_M (R^T+u_{i}u_{\bar i}+u)e^{-u}dV_g=n+\int_M ue^{-u}dV_g.
\]

It follows that
\[
\int_M ue^{-u}dV_g\geq \cW(g(t), u(t), 1)-n\geq \mu(g(0), 1)-n.
\]

It follows that
\[
A(t)=-\int_M ue^{-u}dV_g\leq -\mu(g(0), 1)+n.
\]
\end{proof}

\begin{prop} \label{P-4-2}The transverse scalar curvature $R^T$ satisfies 
\begin{equation}\label{E-4-4}
\frac{\p R^T}{\p t}= \t^T R^T-R^T+|Ric^T|^2.
\end{equation}
In particular, $R^T$ is bounded from below. 
\end{prop}

\begin{proof} A straightforward computation gives \eqref{E-4-4}. Note that $R^T\in C^\infty_B(M)$; then $\t^T R^T=\t R^T$. Hence by the maximum principle 
$R^T$ is bounded from below. 
\end{proof}

\begin{lemma}\label{L-4-3} $u(x, t)$ is uniformly bounded from below.
\end{lemma}

\begin{proof}

By Lemma \ref{L-4-1} and Proposition \ref{P-4-2}, $-R^T+A(t)$ is bounded from above.  Suppose $u(x_0, t_0)$ is very negative at $(x_0, t_0)$ such that
\[u(x_0, t_0)+\max_{x, t} (n-R^T+A)\ll0.\]
 Then we can get that
\begin{equation}\label{E-4-5}
\frac{\p u}{\p t}=n-R^T+u+A<0
\end{equation}
at $(x_0, t_0)$. It follows that $u(x, t)\leq u(x_0, t_0)$ for $t\geq t_0$. Since $u(x, t)$ is smooth, there exists a neighborhood $U$ of $x_0$ such that  for any $x\in U$,
\[u(x, t_0)+\max_{x, t}(n-R^T+A)\ll0.\]
It then follows from \eqref{E-4-5} that $u(x, t)\leq u(x, t_0)$ for any $(x, t)\in U\times [t_0, \infty)$. 
We can  get that
\[
u(x, t)\leq e^{t-t_0}(C+u(x, t_0))\; \mbox{by}\;
\frac{\p u}{\p t}\leq C+u.
\]
Hence for $(x, t)\in U\times [t_0, \infty)$,
\[
u(x, t)\leq -C_1 e^{t},
\]
where $C_1$ depends on $t_0$.
Then $\p \phi/\p t= u$ implies that
\[
\phi(x, t)\leq \phi(x, t_0)-C_1e^{t-t_0}\leq -C_2 e^t
\]
for $t$ big enough and $x\in U$. On the other hand, by the normalized condition
\[
\int_M e^{-u}dV_g=1,
\]
we can get that 
\[
u(x(t), t)=\max_x u(x, t)\geq -C. 
\]
By \eqref{E-4-3}, we get that
\[
\frac{\p}{\p t}(u-\phi)=n-R^T+A\leq C.
\]
Hence
\[
u(x(t), t)-\phi(x(t), t)\leq \max_x (u(x, 0)-\phi(x, 0))+Ct. 
\]
It then follows that
\begin{equation}\label{E-4-6}
\max_x\phi(x, t)\geq -C-Ct. 
\end{equation}
It is clear that $n+\t^T_{g(0)}\phi=n+\t_{g(0)}\phi> 0$. Let $G_0$ be the Green function of $g(0)$ and 
apply  Green's formula to $\phi(x, t)$ with respect to the metric $g(0)$.  We can get for $t\geq t_0$ and any $x$, 
\[
\begin{split}
\phi(x, t)&=\int_M \phi dV_0-\int_M \t_{g(0)}(y, t)G_0(x, y)dV_0(y)\\
&\leq Vol_0(M\backslash U) \max_x \phi(x, t)+Vol_0(U)\int_U\phi(y, t)dV_0(y)+C\\
&\leq Vol_0(M\backslash U) \max_x \phi(x, t)-C_3e^t+C. 
\end{split}
\]
Since $Vol_0(M\backslash U)<Vol_0(M)=1$, we get that
\begin{equation}\label{E-4-7}
\max_x\phi(x, t)\leq -C_4e^t+C_5.
\end{equation}
Note that all constants $C_1, C_2, C_3, C_4, C_5$ are independent of $t\geq t_0$. For $t$ big enough, \eqref{E-4-6} contradicts \eqref{E-4-7}. Hence there exists a constant $B\geq 1$ such that $u(x, t)\geq -B$ for any $(x, t)$. 
\end{proof}

We shall use the notation
\[
\Box f=\left(\frac{\p}{\p t}-\t \right) f.
\]
Note that $u$ is basic; standard computations give that
\begin{equation}\label{E-4-8}\begin{split}
\Box  (\t u)&=-|u_{i\bar j}|^2+\t u,\\
\Box (|u_i|^2)&=-|u_{ij}|^2-|u_{i\bar j}|^2+|u_i|^2. 
\end{split}
\end{equation}

\begin{lemma}
\label{L-4-4}
There is a uniform constant $C$ such that
\begin{equation}\label{E-4-9}
\begin{split}
|u_i|^2&\leq C(u+2B),\\
R^T&\leq C(u+2B). 
\end{split}
\end{equation}
\end{lemma}

\begin{proof}By Lemma \ref{L-4-3}, there is a uniform constant $B\geq 1$ such that $u(x, t)+B>0$. Denote
\[
H=\frac{|u_i|^2}{u+2B}. 
\]
Then straightforward computations (see \eqref{E-4-8}) give that
\begin{equation}\label{E-4-10}
\Box H=\frac{-|u_{ij}|^2-|u_{i\bar j}|^2}{u+2B}+\frac{|u_i|^2(2B-A)}{(u+2B)^2}+\frac{2\mbox{Re} (u_{\bar j} (|u_i|^2)_j)}{(u+2B)^2}-\frac{2|u_i|^4}{(u+2B)^3}.
\end{equation}

We compute
\[
H_j=\frac{(|u_i|^2)_j}{u+2B}-\frac{|u_i|^2 u_j}{(u+2B)^2}.\]
Hence
\begin{equation}\label{E-4-11}
\frac{\mbox{Re} (u_{\bar j} (|u_i|^2)_j)}{(u+2B)^2}-\frac{|u_i|^4}{(u+2B)^3}=\frac{\mbox{Re}(u_{\bar j}H_j)}{u+2B}.
\end{equation}

It is clear that
\[
|u_{\bar j}(|u_i|^2)_j|\leq |u_i|^2(|u_{ij}|+|u_{i\bar j}|).
\]
Then by Cauchy-Schwarz inequality, we compute
\begin{equation}\label{E-4-12}
\frac{|u_{\bar j}(|u_i|^2)_j|}{(u+2B)^2}\leq \frac{|u_i|^4}{2(u+2B)^3}+\frac{|u_{ij}|^2+|u_{i\bar j}|^2}{u+2B}.\end{equation}
Taking \eqref{E-4-10}, \eqref{E-4-11} and \eqref{E-4-12} into account, we can choose $\epsilon<1/4$, such that
\begin{equation}\label{E-4-13}
\Box H\leq \frac{|u_i|^2(2B-A)}{(u+2B)^2}-\frac{\epsilon}{2} \frac{|u_i|^4}{(u+2B)^3}+(2-\epsilon)\frac{\mbox{Re}(u_{\bar j}H_j)}{u+2B}.
\end{equation}
Suppose that
\[
\frac{|u_i|^2}{u+2B}\rightarrow \infty, ~\mbox{when}~ t\rightarrow \infty.
\]
Then there exists time $T$ such that
\begin{equation}\label{E-4-14}
\max_{M\times [0, T]}\frac{|u_i|^2}{u+2B}>2(2B-A)\epsilon^{-1}
\end{equation}
and the maximum is obtained at some point $(p, T)\in M\times [0, T]$. Since at $(p, T)$, 
\[
\Box H\geq 0, H_j=0,
\]
it then follows from \eqref{E-4-13} that
\[
0\leq \Box H\leq \frac{|u_i|^2}{(u+2B)^2}\left(2B-A-\frac{\epsilon}{2} \frac{|u_i|^2}{u+2B}\right),
\]
which contradicts \eqref{E-4-14}. Hence there exists a uniform constant $C$ such that
\[
|u_i|^2\leq C(u+2B). 
\]

Now we prove that $-\t u$ is bounded from above by $C(u+2B)$ for some  $C$. 
Let 
\[
K=\frac{-\t u}{u+2B}, \;\;G=K+\frac{2|u_i|^2}{u+2B}.
\]

Straightforward computations give that
\[
\Box K=\frac{|u_{i\bar j}|^2}{u+2B}+\frac{(-\t u)(2B-A)}{(u+2B)^2}+2\frac{\mbox{Re}(u_{\bar j}K_j)}{u+2B}
\]
and
\[
\Box G=\frac{-2|u_{ij}|^2-|u_{i\bar j}|^2}{u+2B}+\frac{(-\t u+2|u_i|^2)(2B-A)}{(u+2B)^2}+2\frac{\mbox{Re}(u_{\bar j})G_j}{u+2B}.
\]
Hence at the maximum point of $G$,
\[
0\leq \Box G\leq \frac{-|u_{i\bar j}|^2}{u+2B}+\frac{-\t u(2B-A)}{(u+2B)^2}+\frac{2(2B-A)|u_i|^2}{(u+2B)^2}.
\]
We can choose a local coordinate 
\[
(\t u)^2=(u_{i\bar i})^2\leq nu_{i\bar i}^2=n|u_{i\bar j}|^2.
\]
It then follows that
\[
-\frac{|u_{i\bar j}|^2}{u+2B}+\frac{-\t u(2B-A)}{(u+2B)^2}\leq \frac{-\t u}{u+2B}\left(\frac{2B-A}{u+2B}-\frac{-\t u}{n}\right).
\]
By the similar argument (the maximum principle as above), it is clear that
\[
\frac{-\t u}{u+2B}\leq C
\]
for some uniform constant $C$. Since $R^T=n-\t u$, the proof is complete. 
\end{proof}

In view of Lemma \ref{L-4-4}, we need to bound the Ricci potential from above to bound the scalar curvature and $|\nabla u|$; this bound can be obtained once we can bound the diameter of the manifold along the flow.
\begin{prop}\label{P-4-5}There is a uniform constant $C$ such that
\[
\begin{split}
u(y, t)\leq C d_t^2(x, y)+C,\\
R^T(y, t)\leq Cd_t^2(x, y)+C,\\
|\nabla u|\leq Cd_t(x, y)+C,
\end{split}
\]
where $d_t$ is the distance function with respect to $g(t)$ and  $u(x, t)=\min_{y\in M} u(y, t)$. 
\end{prop}
\begin{proof}
We assume that $u+B\geq 0$ for some positive constant $B\geq 1$. Note that
$|\nabla u|^2=|u_i|^2$ since $u$ is basic.
By Lemma \ref{L-4-4}, 
\[
|\nabla \sqrt{u+2B}|\leq C.
\]
Hence 
\[
\sqrt{u(y, t)+2B}-\sqrt{u(z, t)+2B}\leq C d_t(y, z) .
\]
Note that by the normalization condition,
\[
u(x, t)=\min_{y\in M} u(y, t)\leq C,\]
where $C$ is a uniform constant.  

It then follows that
\[
\sqrt{u(y, t)+2B}\leq Cd_t(y, x)+\sqrt{2B}. \]
Hence there is a uniform constant $C$ such that
\begin{equation}\label{E-4-15}
u(y, t)\leq C d_t^2(x, y)+C. 
\end{equation}
The other two estimates are direct consequence of Lemma \ref{L-4-4} and \eqref{E-4-15}. 
\end{proof}
So far the discussions above are pretty much the same as in the K\"ahler setting \cite{Sesum-Tian}. But we shall see that the Sasakian structure plays a crucial in the following discussion.

\section{The Bound on Diameter --- Regular or Quasi-regular Case}

In this section we consider that $(M, \xi, g)$ is regular or quasi-regular as a Sasakian structure.
We shall use the $\cW$ functional on Sasaki manifolds to derive some non-collapsing results along the Sasaki-Ricci flow and then prove that the diameter is uniformly bounded along the Sasaki-Ricci flow. This result can be viewed as generalization of Perelman's results on K\"ahler-Ricci flow to K\"ahler orbifolds which are underlying K\"ahler orbifolds of Sasaki manifolds.  However the smooth Sasakian structure makes it possible to get these results  without detailed discussion of the orbifold singularities on the underlying K\"ahler orbifolds. In particular, one can allow the singular set to be co-dimension 2.   The discussion  follows closely the K\"ahler case \cite{Sesum-Tian} except that  the $\cW$ functional in the Sasaki-Ricci flow  involves only basic functions while the distance function of a Sasakian metric is not basic. To overcome this difficulty, we shall explore the relation of the Sasakian structure with its distance function. 

When  $\xi$ is fixed and the metrics are under deformation generated by  basic potentials $\phi$ such that
\[
\eta_\phi=\eta+d_B^c \phi,
\] 
the distance  along the $\xi$ direction does not change under the deformation. 
Hence we shall introduce the {\it transverse distance}  adapt to the Sasakian structures with $\xi$ fixed  when $M$ is regular or quasi-regular.  Recall that the Reeb vector field $\xi$ defines a foliation of $M$ through its orbits.  
The orbits are compact when $(M, \xi, g)$ is regular or quasi-regular. Moreover, $M$ is a principle $S^1$  bundle (or orbibundle) on a K\"ahler manifold (or orbifold) $Z$ such that $\pi: (M, g)\rightarrow (Z, h)$ is a orbifold Riemannian submersion and 
\[
g=\pi^* h+\eta\otimes \eta.
\]

When $(M, \xi, g)$ is regular (or quasi-regular), the transverse distance is closely related to the distance function on the K\"ahler manifold (or orbifold) $Z$, and the results in the Sasaki-Ricci flow on $M$ can be viewed as the corresponding results in the K\"ahler-Ricci flow on $Z$, even though we shall not deal with the K\"ahler-Ricci flow for the orbifold $Z$ directly.   Nevertheless we can apply the similar arguments as in the Ricci flow and the K\"ahler-Ricci flow, for example, see \cite{Perelman01,  Sesum-Tian, KL}  for more details.

Note that the orbit  $\xi_x$ is compact for any $x$. We can define the {transverse distance} function as
\begin{equation}\label{E-5-1}
d^T_g(x, y):=d( \xi_x,  \xi_y),\end{equation}
where $d$ is the distance function defined by $g$. Note that $d^T_g$ is not a distance function on $M$ and we shall use $d^T(x, y)$ when there is no confusion. We also define the {\it transverse diameter} of a Sasaki structure $(M, \xi, g)$ as
\[
d^T_g=\max_{x, y\in M} d^{T}_g(x, y).
\]

\begin{prop}Let $(M, \xi, g)$ be a regular or quasi-regular Sasaki structure. For any point $p\in \xi_x$, there exists a point $q\in \xi_y$ such that
\[
d^T(x, y)=d(p, q)
\]
In particular, 
\begin{equation}\label{E-3-5d}
d^T(x, y)=d(x, \xi_y)=d(y, \xi_x).
\end{equation}
As a consequence, 
the transverse distance function $d^T$  satisfies the triangle inequality, 
\begin{equation}\label{E-5-2}
d^T(x, z)\leq d^T(x, y)+d^T(y, z).
\end{equation}
Hence for any $p\in M$, $d^T(p, x)$ is a basic Lipschitz function such that \[\langle \xi, \nabla d^T(p, x)\rangle=0.\] Moreover, let $\pi: (M, g)\rightarrow  (Z, h)$ be the (orbifold) Riemannian submersion, we have 
\begin{equation}\label{E-5-d}
d^T(x, y)=d_h(\pi(x), \pi(y)). 
\end{equation}

\end{prop}

\begin{proof}
Proposition \ref{E-5-1} is indeed a consequence of results of Reinhart \cite{Reinhart} and Molino \cite{Molino} on Riemannian foliations (dimension $1$) and bundle-like metrics ; see, for example, Section 3 (Proposition 3.5, Lemma 3.7, Proposition 3.7) in \cite{Molino}.  Molino proved not only the orbifold structure of foliation space $M/\cF_\xi=Z$, but constructed the local (orbifold) coordinates of $(Z, h)$ from some tubular neighborhood of orbits of $\xi$.  
With this construction it is not hard to check that \eqref{E-5-d} holds.

We shall give a proof for completeness. 
Since $\xi_x$ and $\xi_y$ are both compact, there exist $\bar x\in \xi_x$, $\bar y\in \xi_y$, such that 
\[
d(\bar x, \bar y)=d^T(x, y).
\]
Now we want to prove that for any $p\in \xi_x$, there exists $q\in \xi_y$ such that $d(p, q)=d^T(x, y)$. 
Let $\gamma: [0, 1]\rightarrow M$ be the geodesic  such that $\gamma(0)=\bar x$, $\gamma(1)=\bar y$, and the length of $\gamma$ is $d(\bar x, \bar y)$. 
Recall that $\xi$ generates a $S^1$ action on $(M, g)$ and $g$ is invariant under the action. 
For any $p\in \xi_x$, there exists an isomorphism $\lambda: (M, g)\rightarrow (M, g)$, which is an element in $S^1$ generated by $\xi$, such that, $\lambda (\bar x)=p$. Let $q=\lambda (\bar y)$ and $\bar \gamma =\lambda \circ \gamma$. Since $\lambda^{*} g=g$, $\bar \gamma$ is also a geodesic which connects $p$ and $q$ such that the length of $\bar \gamma$ is the same as the length of $\gamma$.  It follows that  $d(p, q)=d^T(x, y)$. As a direct consequence, \eqref{E-3-5d} holds.  

For $x, y, z\in M$, we can choose $o\in \xi_x, p\in \xi_y$ and $q\in \xi_z$  such that
\[
d(o, p)=d^T(x, y)~\mbox{and}\; d(p, q)=d^T(y, z). 
\]
It then follows that
\begin{equation}\label{E-6-triangle}
d^T(x, y)+d^T(y, z)=d(o, p)+d(p, q)\geq d(o, q)\geq d^T(x, z). 
\end{equation}
In particular, \eqref{E-6-triangle}  together with $d^T(x, y)\leq d(x, y)$ for any $x, y\in M$ imply that $d^T$ is a Lipschitz function on $M$. For any $p\in M$ fixed, $f(x)=d^T(p, x)$ is constant for $x$  along the geodesics (the orbits) generated by $\xi$, hence it is a (Lipschitz) basic function.

Now we prove \eqref{E-5-d}. Suppose $p\in \xi_x, q\in \xi_y$ such that $d^T(p, q)=d(p, q)=d(\xi_x, \xi_y)$. Let $\gamma(t)$ be one shortest geodesic such that $\gamma(0)=p, \gamma(1)=q$. By the first variation formula of  geodesics,  we have $\langle\dot \gamma(1), \xi\rangle=0$ at $T_qM$. By Proposition 1 in \cite{Reinhart}, $\dot \gamma(t)\perp \xi$ for any $t$.  Hence $\gamma(t)$ is an orthogonal geodesic such that the projection  $\pi(\gamma(t))$ is also a geodesic in $(Z, h)$, see \cite{Molino} Proposition 3.5 for example.  
It is clear that $\gamma$ only intersects any orbit of $\xi$ at most once, it follows that $\pi(\gamma)$ does not intersect with itself. Moreover the length of $\gamma$ (with respect to $g$) is equal to the length of $\pi(\gamma)$ with respect to $h$.  This gives $d^T(p, q)\geq d_h(\pi(p, \pi(q)))$. 

Any geodesic in $(Z, h)$ can locally be lifted up to an orthogonal geodesic on $(M, g)$ which has the same length.  
Suppose for $\pi(p)\neq \pi(q)\in Z$, let  $\gamma(t)$ be a shortest geodesic connecting them. We can divide $\gamma$ into short segments such that each segment has a lift (an orthogonal geodesic) in $(M, g)$. Note that  any orthogonal geodesic in $(M, g)$ can be slid along the orbit of $\xi$. So we can construct a lift of $\gamma$ from the local lifts by sliding them. This gives  a piece-wise curve  in $M$ which connects $p$ and a point $\tilde q$ in $\xi_q$ and it has the same length as $\gamma$. This gives  $d^T(p, q)\leq d_h(\pi(p), \pi(q))$. Hence it completes the proof. 
 \end{proof}

We can define a {\it transverse ball} $B_{\xi, g}(x, r)$ as follows,
\begin{equation}\label{L-5-1a}
B_{\xi, g}(x, r)=\{y: d^T(x, y)<r\}.
\end{equation}
By Proposition \ref{E-5-1}, we have
\begin{equation}
B_{\xi, g}(x, r)=\{y: d(\xi_x, y)<r\}. 
\end{equation}

The following non-collapsing theorem for a transverse ball  holds along the Sasaki-Ricci flow,  similar to the Ricci flow (K\"ahler-Ricci flow).

\begin{lemma}\label{L-5-2}
Let $(M, \xi, g_0)$ be a regular or quasi-regular Sasaki structure and let $g(t)$ be the solution of the Sasaki-Ricci flow with the initial metric $g_0$. Then there exists a positive constant $C$ such that for every $x\in M$, if $R^T\leq Cr^{-2}$ on $B_{\xi, g(t)}(x, r)$ for $r\in (0, r_0]$, where $r_0$ is a fixed sufficiently small positive number, then 
\[
Vol(B_{\xi, g(t)}(x, r))\geq Cr^{2n}. 
\]
\end{lemma}

\begin{proof}Let $g(t)$ be the solution of Sasaki-Ricci flow. Since $R^T$ is bounded from below, we can assume that $r_0$ is small enough such that $R^T(x, t)\geq -r_0^{-2}$ for any $x, t$.
We argue  by contradiction; suppose the result is not true, then there exists  sequences $(p_k, t_k)\in M\times [0, \infty)$ and $t_k\rightarrow \infty$ such that $R^T\leq Cr_k^{-2}$ in 
$B_k=B_{\xi, g(t_k)}(p_k, r_k)$, but $Vol(B_k)r_k^{-2n}\rightarrow 0$ as $k\rightarrow \infty$. Let $\Phi$ be the
cut-off function $\Phi: [0, \infty)\rightarrow [0, 1]$ such that  $\Phi$ equals $1$ on $[0, 1/2]$, decreases on $[1/2, 1]$ with derivative bounded by $4$ and equals $0$ on $[1, \infty)$. 
Denote $\tau_k=r_k^2$ and define
\[
w_k=e^{C_k}\Phi(r_k^{-1}d^T_k(p_k, x)),
\]
where $d^T_k$ is the transverse distance with respect to time $t_k$ and $C_k$ is the constant such that
\[
1=\int_M w_k^2 \tau_k^{-n}dV_k=e^{2C_k}r_k^{-2n}\int_M \Phi^2 (r_k^{-1}d^T_k(p_k, x))dV_k\leq e^{2C_k}r_k^{-2n}Vol(B_k),
\] 
where we use $dV_k$ to denote the volume element with respect to $g(t_k)$.
Hence $C_k\rightarrow \infty$ since $Vol(B_k)r_k^{-2n}\rightarrow 0$ when $k\rightarrow \infty$. 
We compute (with $g_k=g(t_k)$)
\begin{equation}\label{E-5-3}
\begin{split}
\cW_k&= \tau_k^{-n}\int_M (r_k^2 (R^T w_k^2+4|\nabla^T w_k|^2)-w_k^2\log w_k^2)dV_k\\
&\leq e^{2C_k}r_k^{-2n}\int_M (4|\Phi^{'}_k|^2-\Phi_k^2\log \Phi_k^2)dV_k+r_k^2\max_{B_k}R^T-2C_k,
\end{split}
\end{equation}
where $\cW_k=\cW(g_k, w_k, \tau_k)$ and $\Phi_k=\Phi(r_k^{-1}d^T_k(p_k, x))$.
Let
\[
B_k(r)=B_{\xi, g_k}(p_k, r),\; \mbox{and}\; V_k(r)=Vol(B_k(r)). 
\]
 For any $k$ fixed,
\begin{equation}\label{E-volume}
\lim_{r\rightarrow 0} \frac{V_k(r)}{V_k(r/2)}=2^{2n}=4^n. 
\end{equation}
The proof of \eqref{E-volume} can run as follows. Recall that $\pi: M\rightarrow Z$ is a Riemannian submersion over the orbifold $Z$ and $M$ is the $S^1$ principle orbibundle over $Z$.  For any $p_k\in M$,   $\pi(p_k)\in \pi(B_k(r))\subset Z$ and $\pi(B_k(r))$ is the geodesic $r$ ball centered at $\pi(p_k)$ with respect to $h_k$.  Hence when $r$ small enough, $B_k(r)$ is a trivial $S^1$ bundle over the geodesic $r$ ball $B_{h_k}(r)$ of $(Z, h_k)$ centered at $\pi(p_k)$. $\pi(p_k)$ can be either a smooth point or an orbifold singularity in $Z$. Nevertheless, we have, 
\[
\lim_{r\rightarrow 0}\frac{Vol(B_{h_k}(r))}{Vol(B_{h_k}(r/2))}=2^{2n}=4^n. 
\]
Note that the orbifold singularities in $Z$ is a measure zero set and does not contribute when we compute volume. 
Without loss of generality, we can assume that $B_{h_k}(r)$ contains only smooth points. Otherwise, we can only consider the set $\Sigma$ of smooth points in $B_{h_k}(r)$. 
For generic points (smooth points on $Z$), any $S^1$ fibre has the same length. We denote the length of a generic $S^1$ fibre by $l$.  Note that we have  
\[
g_k=\pi^*h_k+\eta_k\otimes \eta_k, 
\]
and $dvol_{g_k}=\pi^{*}(dvol_{h_k})\wedge \eta_k$.  It follows that
\[Vol_k(B_k(r))=\int_{B_k(r)}dvol_{g_k}=\int_{B_{h_k}(r)\times S^1} \pi^{*}(dvol_{h_k})\wedge \eta_k= l\times Vol(B_{h_k}(r))). \]
This proves \eqref{E-volume}.

We can assume that, in addition, 
\begin{equation}\label{E-5-4}
V_k(r_k) \leq 5^n V_k(r_k/2).
\end{equation}
Otherwise, 
 let $r_k^{i}=2^{-i}r_k$ for $i\in \mathbb{N}$. By \eqref{E-volume}, we can choose $i_0$ to be the smallest number such that
\[
V_k(r_k^{i_0})\leq 5^n V_k(r_k^{i_0}/2).
\]

Hence for any $i\leq  i_0-1$,
\begin{equation}\label{E-0-6}
V_k(r_k^i)>5^nV_k(r_k^{i}/2)=5^nV_k(r_k^{i+1}). 
\end{equation}
By using \eqref{E-0-6} repeatedly, we can get that 
\[
\left(r_k^{i_0}\right)^{-2n}V_k(r_k^{i_0})\leq \left(\frac{4}{5}\right)^{ni_0} r_k^{-2n}V_k(r_k)\rightarrow 0
\]
when $k\rightarrow \infty$. Hence we can replace $r_k$ by $r_k^{i_0}$, which satisfies \eqref{E-5-4}  in addition. We then compute, by \eqref{E-5-4},
\[\begin{split}
\int_M (4|\Phi^{'}_k|^2-\Phi_k^2\log \Phi_k^2)dV_k&\leq C(V_k(r_k)-V_k(r_k/2))\\
&\leq C5^nV_k(r_k/2)\\
&\leq C5^n \int_M \Phi_k^2 dV_k.
\end{split}
\]
Hence we compute, by \eqref{E-5-3} and the above,
\[
\cW_k\leq C5^n\tau_k^{-n}\int_M w_k^2dV_k+C-2C_k\leq C-2C_k.
\]
Note that $\mu(g(t), 1+(\tau_0-1)e^t)$ is non-decreasing function on $t$.  Choose $\tau_0^k=1-(1-r_k^2)e^{-t_k}$, then we get that
\[
\mu(g(0), \tau_0^k)\leq \mu(g(t_k), r_k^2)\leq \cW_k\rightarrow-\infty. 
\]
Contradiction since $\mu(g(0), \tau)$ is a continuous function of $\tau$ in $(0, \infty)$ and $\tau_0^k\rightarrow 1$ when $k\rightarrow \infty$. 

\end{proof}

Lemma \ref{L-5-2} is not purely local (it is global in $\xi$ direction); while $\xi$ generates isometries of $g$, one can further obtain local non-collapsing results and it should be viewed as the corresponding non-collapsing result of the K\"ahler-Ricci flow on K\"ahler orbifolds.  Now we can bound the diameter of the manifold along the Sasaki-Ricci flow. 

\begin{theo}\label{T-5-3}Let $(M, \xi, g_0)$ be a regular or quasi-regular Sasaki structure and let $g(t)$ be the solution of the Sasaki-Ricci flow with the initial metric $g_0$. 
Then the transverse diameters $d^T_{g(t)}$ are uniformly bounded. As a consequence there is a uniform constant $C$ such that $diam(M, g(t))\leq C$. 
\end{theo}

\begin{proof}We argue by contradiction.  Note that  the orbits of $\xi$ are closed geodesics and the length of these closed geodesics  does not change along the flow, hence it is uniformly bounded. Suppose the generic orbits have length $l=l(M, \xi, g_0)$.  Then for any $x, y\in M$,
\begin{equation}\label{E-5-6}
d_t(x, y)\leq d^T_t(x, y)+l,
\end{equation}
where $d_t$ and $d^T_t$ are the distance function and transverse distance with respect to $g(t)$. 
Hence we only need to prove that the transverse diameter $d^T_t$ is uniformly bounded along the flow $g(t)$. Let $u(x, t)$ be the transverse Ricci potential defined in \eqref{E-4-1}. Recall that $u(x, t)$ is uniformly bounded from below by Lemma \ref{L-4-3}. Choose a point   $x_t\in M$ such that
\[
u(x_t, t)=\min_{y\in M} u(y, t). 
\]
Denote 
\[
d_t(y)=d_t(x_t, y),\; d_t^T(y)=d_t^T(x_t, y). 
\]
Let $B_{\xi}(k_1, k_2)=\{y: 2^{k_1}\leq d_t^T(y)\leq 2^{k_2}\}.$ Consider the transverse annulus $B_{\xi}(k, k+1)$ $k\geq 0$. By Proposition \ref{P-4-5} and \eqref{E-5-6}, we have
$R^T \leq C2^{2k}$ on $B_{\xi}(k, k+1)$. The transverse annulus $B_{\xi}(k, k+1)$ contains at least $2^{2k-1}$ transverse ball of radius $2^{-k}$ which are not intersected with each other.  When $k$ is large enough, by Lemma \ref{L-5-2}, we have (with $r=2^{-k}$)
\[
Vol(B_{\xi}(k, k+1))\geq \sum_i Vol(B_{\xi, g(t)}(p_i, 2^{-k}))\geq 2^{2k-1}2^{-kn} C,
\]
where $\{p_i\}$ are centers of $2^{2k-1}$ transverse balls contained in $B_{\xi}(k, k+1)$. 

\begin{claim}\label{C-5-4}For every $\epsilon>0$, there exists $k_1, k_2$ such that $k_2-k_1\gg1$ such that if $d^T_t$ is large enough, then

\begin{enumerate}
\item 
$Vol(B_{\xi}(k_1, k_2))<\epsilon,$

\item $
Vol(B_{\xi}(k_1, k_2))\leq 2^{10 n} Vol (B_{\xi}(k_1+2, k_2-2)). 
$

\item There exists $r_1, r_2$ and a uniform constant $C$ such that $r_1\in [k_1 , k_1+1],  r_2\in [k_2-1, k_2]$ and
\[
\int_{B_{\xi}(r_1, r_2)} R^T dV\leq C Vol (B_{\xi}(k_1, k_2)). 
\]

\end{enumerate}

\end{claim}

Assume that the diameter of $(M, g(t))$ is not uniformly bounded in $t$. Hence there exists a sequence $t_i\rightarrow \infty$ such that $d_{t_i}^T\rightarrow \infty$. Let $\epsilon_i\rightarrow 0$ be a sequence of positive numbers. By Claim \ref{C-5-4}, there exist sequences $k_1^i, k_2^i$ such that
\begin{equation}\label{E-5-7}
\begin{split}
V_i(k_1^i, k_2^i)&:= Vol_{t_i}(B_{\xi, t_i}(k_1^i, k_2^i))<\epsilon_i,\\
V_i(k_1^i, k_2^i)&\leq  2^{10n}V_i(k_1^i+2, k_2^i-2).
\end{split}
\end{equation}
For each $i$, we can also find $r_1^i$ and $r_2^i$ as in Claim \ref{C-5-4}. Let $\Phi_i$, for each $i$,  be a cut-off function such that $\Phi_i(t)=1$ for $t\in [2^{k_1^i+2}, 2^{k_2^i-2}]$ and equals zero for
$t$ in $(-\infty, 2^{r_1^i}] \cup [2^{r_2^i}, \infty)$ with the derivative bounded by $2$.  Define $w_i(y)=e^{C_i}\Phi_i(d_{t_i}^T(x_i, y))$, where $x_i=x_{t_i}$ such that
\[
\int_M w_i^2dV_{i}=1. 
\]
This implies that
\[
1=\int_M w_i^2dV_{i}=e^{2C_i}\int_M \Phi_i^2dV_i\leq e^{2C_i} V_i(k_1^i, k_2^i)=e^{2C_i}\epsilon_i. 
\]
It follows that $C_i\rightarrow \infty$ when $i\rightarrow \infty$. 
We compute
\begin{equation}\label{E-5-8}
\begin{split}
\cW_i&=\cW(g(t_i), w_i, 1)\\
&\leq e^{2C_i}\int_M (4|\Phi_i^{'}|^2-\Phi_i^2\log \Phi_i^2)dV_i+\int_M R^T w_i^2dV_i-2C_i. 
\end{split}\end{equation}
 First of all we have, by Claim \ref{C-5-4} (see \eqref{E-5-7}), 
\[
\begin{split}
\int_M R^T w_i^2dV_i&\leq e^{2C_i}\int_{B_{\xi, t_i}(r_1^i, r_2^i)} R^T dV_i\\
&\leq e^{2C_i} C V_i(k_1^i, k_2^i)\\
&\leq e^{2C_i} C 2^{10n} V_i(k_1^i+2, k_2^i-2)\\
&\leq C2^{10n}\int_M w_i^2dV_i\\
&=C2^{10n}. 
\end{split}
\]
Similarly, we have
\[
\begin{split}
e^{2C_i}\int_M (4|\Phi_i^{'}|^2-\Phi_i^2\log \Phi_i^2)dV_i&\leq Ce^{2C_i} V_i(k_1^i, k_2^i)\\
&\leq e^{2C_i} C 2^{10n} V_i(k_1^i+2, k_2^i-2)\\
&\leq C2^{10n}\int_M w_i^2dV_i\\
&=C2^{10n}. 
\end{split}
\]
Hence we get 
\[
\cW_i\leq C-2C_i
\]
for some uniform constant $C$. By the monotonicity of $\mu(g(t), 1)$ (with $\tau_0\equiv 1$), we get that
\[
\mu(g(0), 1)\leq \mu(g(t_i), 1)\leq \cW_i \leq C-2C_i.
\]
Contradiction since $C_i\rightarrow \infty$. Therefore, there is a uniform bound on the diameter of $(M, g(t))$. 
\end{proof}

Now we prove Claim \ref{C-5-4}, hence finish the proof of Theorem \ref{T-5-3}.
\begin{proof}Note that the volume of $M$ remains a constant along the flow. If the diameter of $(M, g(t))$ is not uniformly bounded, then for any $\epsilon>0$, if the diameter of $(M, g(t))$ large enough,  there exists $K$ big enough such that for all $k_2> k_1\geq K$, $Vol(B_{\xi}(k_1, k_2))<\epsilon$. If the estimate (2) does not hold,  we have
\[
Vol(B_{\xi}(k_1, k_2))\geq 2^{10n}Vol(B_{\xi}(k_1+2, k_2-2)).
\]
We then consider $B_{\xi}(k_1+2, k_2-2)$ instead and ask whether (2) holds or not for $B_{\xi}(k_1+2, k_2-2)$. We repeat this step if we do not find two numbers such that (1) and (2) are both satisfied. Suppose for every $p$, at $p$-th step we cannot find two numbers. Then we would have
\[
Vol(B_{\xi}(k_1, k_2))\geq 2^{10np}Vol(B_{\xi}(k_1+2p, k_2-2p)).
\]
We can assume that
\[
k_1+2p\leq k, k_2-2p\geq k+1.\]
In particular, we can choose $k\gg 1$ such that
\[
k_1=\frac{k}{2}, k_2=\frac{3k}{2},\;\mbox{and}\;  p=\left[\frac{k}{4}\right]-1. 
\]
It then follows that
\[
\epsilon>Vol(B_{\xi}(k_1, k_2))\geq 2^{10np} Vol(B_{\xi}(k, k+1))\geq 2^{10np} C 2^{-2nk} 2^{2k-1}.
\]
Note that $10p>2k$ when $k$ is big. This gives a contradiction if $k$ is big enough. 

We define a transverse metric sphere $S_\xi(x, r)$ as
\[
S_{\xi}(x, r)=\{y: d^T_t(x, y)=r\}. 
\]
Hence we have
\[
\frac{d}{dr} Vol(B_{\xi, g}(x, r))=Vol(S_{\xi}(x, r)). 
\]

Given $k_1\ll k_2$ in (1) and (2), there exists $r_1\in [k_1, k_1+1]$ such that for $r=2^{r_1}$,
\[
Vol(S_{\xi}(x, r))\leq \frac{2Vol(B(k_1, k_2))}{2^{k_1}}. 
\]
Suppose not, we have
\[
Vol(B(k_1, k_1+1))=\int_{2^{k_1}}^{2^{k_1+1}} Vol(S_{\xi}(x, r))dr\geq  2Vol(B_{\xi}(k_1, k_2)). 
\]
Contradiction since $k_1\ll k_2$. Similarly there exists $r_2\in [k_2-1, k_2]$ such that
\[
Vol(S_{\xi}(x, 2^{r_2}))\leq  \frac{2Vol(B(k_1, k_2))}{2^{k_2}}.  
\]

We have the estimate, 
\begin{equation}\label{E-i-1}
\begin{split}
&\int_{B_{\xi}(r_1, r_2)}R^T-nVol({B_{\xi}(r_1, r_2)})=\int_{B_{\xi}(r_1, r_2)} (R^T-n)\\
\quad&= -\int_{B_{\xi}(r_1, r_2)} \t u\\
\quad&\leq \int_{S_{\xi}(x, 2^{r_1})} |\nabla u|+\int_{S_{\xi}(x, 2^{r_2})} |\nabla u|\\
\quad&\leq 2CVol(B_{\xi}(k_1, k_2))\left(\frac{1}{2^{k_1}} 2^{k_1+1} +\frac{1}{2^{k_2}} 2^{k_2}\right)\\
\quad&\leq C Vol(B_{\xi}(k_1, k_2)),
\end{split}
\end{equation}
where we have used the estimate of $|\nabla u|$ in Proposition \ref{P-4-5}. 
\end{proof}

\begin{rmk}Note that the transverse distance function is Lipschitz, hence the integration by parts in \eqref{E-i-1} is justified. 
\end{rmk}

\section{Irregular Sasakian Structure and Approximation}
When a Sasakian structure $(M, \xi, \eta, \Phi, g)$ is irregular,  Rukimbira \cite{Rukimbira95a} proved that one can always approximate an irregular Sasakian structure by a sequence of quasi-regular Sasakian structures $(M, \xi_i, \eta_i, \Phi_i, g_i)$ (\cite{Rukimbira95a}, see Theorem 7.1.10, \cite{BG} also). In particular, one can choose $\xi_i=\xi+\rho_i$ for $\rho_i$ in a commutative Lie algebra  spanned by the vector fields defining $\bar \cF_\xi$, the closure of $\cF_\xi$; moreover $\lim \rho_i\rightarrow 0$ and $\eta_i$, $\Phi_i$, $g_i$ can be expressed in terms of $\rho_i$; see \cite{BG} Theorem 7.1.10 for  more details. 
We can then study the behavior of the Sasaki-Ricci flow for each quasilinear Sasakian metric $(M, \xi_i, \eta_i, K_i, g_i)$. 
Note  that the transverse diameter of $g_i(t)$ is uniformly bounded along the Sasaki-Ricci flow for any $t$ and $i$ large enough,  by Theorem \ref{T-5-3}.  Since $R^T$ and $\nabla u$ are both basic, 
 the approximation above is then  enough to bound the transverse scalar curvature $R^T$ and $\nabla u$. 
Note that  the generic orbits for $(M, \xi, K, g)$ are not close,  hence the length of the generic orbits of $(M, \xi_i, K_i, g_i)$  do not have a uniform bound  when $i\rightarrow \infty$. However
 for a compact Sasakian manifold (or K-contact) manifold, there always exists some close orbits of $\xi$ (\cite{Banyaga90}). Under the Sasaki-Ricci flow,  the length of these closed orbits does not change. Hence we can get some  close orbit of $\xi$ on $(M, \xi_i, K_i, g_i)$ under approximation which has a uniformly bounded length. This suffices to prove that the diameter is uniformly bounded  along the Sasaki-Ricci flow  for irregular Sasaki structure.

The main result in this section is as follows,
\begin{theo}\label{T-6-1}Let $(M, \xi, g_0)$ be a  Sasakian structure and let $g(t)$ be the solution of the Sasaki-Ricci flow with the initial metric $g_0$. There is a uniform constant $C$ such that $R^T_0(t)\leq C, |\nabla u_0(t)|\leq C$ and the diameter is also uniformly bounded.  \end{theo}

To prove this result, first let us recall the following proposition proved in \cite{SWZ}.
\begin{prop}[Smoczyk-Wang-Zhang, \cite{SWZ}] \label{P-6-2}Let $g(t)$ be the solution of the Sasaki-Ricci flow with initial metric $g_0$. Then  for any $T\in (0, \infty)$, there is a constant $C$ such that
\[
\|\phi(t)\|_{C^k(M, g_0)}\leq C, \|g(t)\|_{C^k}\leq C
\]
where $C$ depends on $g_0, T, k$. 
\end{prop}

We can prove a  stability result of the Sasaki-Ricci flow under approximation using the above estimate.

\begin{prop}\label{P-6-3}Let $(M, \xi_0, \eta_0, \Phi_0, g_0)$ be a Sasakian structure and $g_0(t)$ be the solution of the Sasaki-Ricci flow with initial metric $g_0$. Suppose $(M, \xi_j, \eta_j, \Phi_j, g_j)$ be a sequence of Sasakian structures such that $(M, \xi_j, \eta_j, \Phi_j, g_j)\rightarrow (M, \xi_0, \eta_0, \Phi_0, g_0)$ in $C^\infty$ topology when $j\rightarrow \infty$ (we  choose a sequence as in the proof of Theorem 7.1.10, \cite{BG}). Then for any $\epsilon$ and $T\in (0, \infty)$, there is a constant $N=N(\epsilon, k, T)$ such that when $j\geq N$, $t\in [0, T]$, 
\[
\|g_j(t)-g_0(t)\|_{C^k(M, g_0)}\leq \epsilon. 
\] 
\end{prop}
\begin{proof}The statement is a standard stability result for parabolic equations. Suppose $g_k(t)$ is the solution of the Sasaki-Ricci flow for the initial metric $g_k$, and  the corresponding potential $\phi_k(t)$ solves \eqref{E-p-1} with $\phi_k(0)=0$. We should rewrite \eqref{E-p-1}, for each $k$, as 
\begin{equation}\label{E-p-2}
\frac{\p \phi_k}{\p t}=\log \frac{\eta_k\wedge (d\eta_k-d\circ \Phi_k \circ d\phi_k)^n}{\eta_k\wedge (d\eta_k)^n}+\phi_k-F_k. 
\end{equation}
We can assume that $F_k\rightarrow F_0$ in $C^\infty$ topology. Note that, for each $k$, $\xi_k$ and the transverse complex structure are fixed; hence $d\circ \Phi_k(t) \circ d\phi_k=d\circ \Phi_k \circ d\phi_k=-2\sqrt{-1}\p_B\bar \p_B\phi_k$. Moreover, by Proposition \ref{P-6-2}, we can assume that $\|\phi_k(t)\|_{C^l}\leq C(l, T, g_0)$ for $t\in [0, T]$. Hence we only need to show that $|\phi_k(t)-\phi_0(t)|\rightarrow 0$ when $k\rightarrow \infty$ for any $t$; then by the standard interpolation inequalities we can then get $\|\phi_k(t)-\phi_0(t)\|_{C^l}\rightarrow 0$ when $k\rightarrow \infty$. To prove $|\phi_k(t)-\phi_0(t)|\rightarrow 0$, we consider
\begin{equation}\label{E-p-3}
\frac{\p}{\p t}  \left(\phi_k(t)-\phi_0(t)\right)=\log \frac{\eta_0 \wedge (d\eta_0-d\circ \Phi_0 \circ d\phi_k(t))^n}{\eta_0\wedge (d\eta_0-d\circ \Phi_0 \circ d\phi_0(t))^n}+\phi_k(t)-\phi_0(t)+f_k,
\end{equation}
where \[f_k=\log \frac{\eta_k \wedge (d\eta_k-d\circ \Phi_k \circ d\phi_k(t))^n}{\eta_0\wedge (d\eta_0-d\circ \Phi_0 \circ d\phi_k(t))^n}-\log \frac{\eta_k\wedge (d\eta_k)^n}{\eta_0\wedge (d\eta_0)^n}+F_0-F_k.\]
Note that $\left\|\phi_k(t)\right\|_{C^l}$ is uniformly bounded (depending only on $T, g_0, l$), $f_k\rightarrow 0$ when $k\rightarrow \infty$. Let $\max |f_k|=c_k$. Let $w(t)=\max_M (\phi_k(t)-\phi_0(t))$. For any $p\in M$ such that $\phi_k(t)-\phi_0(t)$ obtains its maximum at $p$, then at $p$, 
\begin{equation}\label{E-p-4}
\log \frac{\eta_0 \wedge (d\eta_0-d\circ \Phi_0 \circ d\phi_k(t))^n}{\eta_0\wedge (d\eta_0-d\circ \Phi_0 \circ d\phi_0(t))^n}\leq 0. 
\end{equation}
We can choose a local coordinate $(x, z_1, \cdots z_n)$ such that at $p$, $\xi_0=\p_x, \eta_0=dx, d\eta_0=2\sqrt{-1}g^T_{i\bar j}dz_i\wedge dz_{\bar j}$. We 
compute 
\[
\eta_0\wedge \left(d\eta_0-d\circ \Phi_0 \circ d\phi_0(t)\right)^n=2^nn!(\sqrt{-1})^n\det\left(g^T_{i\bar j}+\frac{\p^2\phi_0(t)}{\p z_i \p z_{\bar j}}\right) dx\wedge dZ\wedge d\bar Z.
\]
Note that $\phi_k(t)$ is not a basic function for $\xi_0$, but we can compute 
\[
-d\circ \Phi_0 \circ d\phi_k(t)=2\sqrt{-1}\frac{\p^2 \phi_k(t)}{\p z_i\p z_{\bar j}}dz_i\wedge dz_{\bar j}+\theta \wedge dx,
\]
where $\theta$ is a 1-form. It then follows that
\[
\eta_0 \wedge (d\eta_0-d\circ \Phi_0 \circ d\phi_k(t))^n=2^nn!(\sqrt{-1})^n\det\left(g^T_{i\bar j}+\frac{\p^2\phi_k(t)}{\p z_i \p z_{\bar j}}\right) dx\wedge dZ\wedge d\bar Z.
\]
Note that the hessian of $\phi_k(t)-\phi_0(t)$ is nonpositive at $p$, in particular, we have
\[
\det\left(g^T_{i\bar j}+\frac{\p^2\phi_k(t)}{\p z_i \p z_{\bar j}}\right)\leq \det\left(g^T_{i\bar j}+\frac{\p^2\phi_0(t)}{\p z_i \p z_{\bar j}}\right).\]Hence \eqref{E-p-4} is confirmed. By the  standard maximum principle argument and \eqref{E-p-3}, we can then get 
\[
\frac{\p w(t)}{\p t}\leq w(t)+c_k.
\]
Hence $w(t)\leq c_k (e^T-1)$. Similarly, we can prove that $-c_k(e^T-1)\leq \phi_k(t)-\phi_0(t)\leq c_k(e^T-1)$. 
\end{proof}

Now we are in the position to prove Theorem \ref{T-6-1}. 

\begin{proof}Pick up a sequence of quasi-regular Sasakian structure $(M, \xi_j, g_j)$ such that $(M, \xi_j, g_j)\rightarrow (M, \xi, g_0)$. Then by Theorem \ref{T-5-3}, the transverse diameter $d^T(g_j(t))$ is uniformly bounded for any $j$ and $t$. 
By Lemma \ref{L-4-4}, we just need to bound $u_j(t)$, the normalized Ricci potential, from above as in Proposition \ref{P-4-5}.  The only difference is that we need to use the fact that $u_j$ is basic and use the transverse diameter to replace the diameter in Proposition \ref{P-4-5}. We can sketch the proof as follows. 
By Lemma \ref{L-4-4}, we know that
\[
|\nabla \sqrt{u_j+2B}|\leq C.
\]
Since $u_j$ is basic with respect to $(M, \xi_j, g_j)$, then for any $y, z$
\[
\sqrt{u_j(y, t)+2B}-\sqrt{u_j(z, t)+2B}\leq C d^T_{g_j(t)}(y, z). 
\]
Let $u_j(x, t)=\min_M u(y, t)$. Note that $u_j(x, t)\leq C$ for some uniformly bounded constant $C$ by the normalized condition  \eqref{E-4-2}. It implies that
\[
\sqrt{u_j(y, t)+2B}\leq Cd^T_{g_j(t)}(x, y)+C.
\]
This gives a uniform upper bound for $u_j(t)$,  hence a uniform upper bound for $R^T_j$ and $|\nabla u_j|$. By Proposition \ref{P-6-3},  for any $t\in (0, \infty)$, $g_j(t)\rightarrow g_0(t)$ for $j$ large enough, hence $R^T_j\rightarrow R^T_0, |\nabla u_j|\rightarrow |\nabla u_0|$. 

Now we prove that the diameter is uniformly bounded along the Sasaki-Ricci flow. 
Recall that  there are  at least $n+1$ close orbits on a compact K-contact manifold (\cite{Rukimbira95, Rukimbira99} ). Hence we can suppose that $(M, \xi, g)$ has a close orbit $\cO$ of $\xi$. Since $(M, \xi_j, g_j)\rightarrow (M, \xi, g_0)$, we can find an orbit $\cO_i$ of $\xi_i$ on $(M, \xi_i, g_i)$ whose length  converges to the length of $\cO$, hence the length $\cO_i$ is uniformly bounded.
Note that the transverse diameter is also uniformly bounded along the Sasaki-Ricci flow for $(M, \xi_i, g_i)$. This proves that the diameter is also uniformly bounded along the Sasaki flow for $(M, \xi_i, g_i)$. By Proposition \ref{P-6-3} again, this implies that the diameter is uniformly bounded along the Sasaki-Ricci flow for $(M, \xi, g)$. 

 \end{proof}

\section{Compact Sasakian Manifolds of Positive Transverse Holomorphic Bisectional Curvature}

In this section we consider compact Sasakian manifolds with positive transverse holomorphic bisectional curvature. Transverse holomorphic bisectional curvature  can be defined as the holomorphic bisectional curvature for transverse K\"ahler metric of a Sasakian metric and it was studied in a recent interesting paper \cite{Zhang}. We recall some  definitions.

\begin{defi}Given two $J$invariant planes $\sigma_1, \sigma_2$ in $\cD_x\subset T_xM$, the holomorphic bisectional curvature $H^T(\sigma_1, \sigma_2)$ is defined as
\[
H^T(\sigma_1, \sigma_2)=\langle R^T(X, JX)JY, Y\rangle,
\]
where $X\in \sigma_1, Y\in \sigma_2$ are both unit vectors. 
\end{defi}

It is easy to check that $\langle R^T(X, JX)JY, Y\rangle$ depends only on $\sigma_1, \sigma_2$, hence $H^T(\sigma_1, \sigma_2)$ is well defined. 
By the first Bianchi identity, one can check that
\[
\langle R^T(X, JX)JY, Y\rangle=\langle R^T(Y, X)X, Y\rangle+\langle R^T(Y, JX)JX, Y\rangle. 
\]
It is often convenient to treat (transverse) holomorphic bisectional curvature in (transverse) holomorphic coordinates. 
Suppose $u, v\in \cD$ are two unit vectors and let $\sigma_u, \sigma_v$ be two $J$-invariant planes spanned by $\{u, Ju\}$ and $\{v, Jv\}$ respectively. Set
\[
U=\frac{1}{2}\left(u-\i Ju\right), V=\frac{1}{2}\left(v-\i Jv\right). 
\] 
Denote $R^T(V, \bar V; U, \bar U)=\langle R^T(V, \bar U)U, \bar V\rangle$, then we have
\begin{equation}\label{E-8-1}
R^T(V, \bar V; U, \bar U)=\frac{1}{4}H^T(\sigma_u, \sigma_v). 
\end{equation}

\begin{defi}\label{D-8-2} For $x\in M$,  the transverse holomorphic bisectional curvature $H^T$ is positive (or nonnegative) at $x$ if $H^T(\sigma_1, \sigma_2)$ is positive for any two $J$ invariant planes $\sigma_1, \sigma_2$ in $\cD_x$. $M$ has positive (nonnegative) transverse holomorphic bisectional curvature if $H^T$ is positive (nonnegative) at any point $x\in M$; equivalently, positivity of transverse holomorphic bisectional curvature  means that $R^T(V, \bar V; U, \bar U)> 0$ for any $V, U\in\cD_x\subset T_xM$ and all $x\in M$. For simplicity, we shall also use the notations $R^T\geq 0$ and $R^T>0$. If a tensor field $S$ has the exact same type as $R^T$, we can also define its positivity in the same way. If $S$ and $T$ have the same type as $R^T$, then $S\geq T$ if and only if $S-T\geq0$. 
\end{defi}

It is also useful to consider the transverse holomorphic bisectional curvature locally. 
Recall that there is an open cover $\{U_\alpha\}$ of $M$,  $V_\alpha\subset \C^n$, submersions $\pi_\alpha: U_\alpha\rightarrow V_\alpha$ such that
\[
\pi_\alpha\circ \pi^{-1}_\beta: \pi_\beta(U_\alpha\cap U_\beta)\rightarrow \pi_\alpha (U_\alpha\cap U_\beta)
\]
is biholomporphic on $U_\alpha\cap U_\beta$. And the transverse metric $g^T$ is well defined on each $U_\alpha$. Note that $g^T_\alpha$ is a genuine  K\"ahler metric on $V_\alpha$ and we can identify the transverse holomorphic bisectional curvature of the Sasakian metric $g$ on $U_\alpha$ with the holomorphic bisectional curvature of $g^T_\alpha$ on each $V_\alpha$.  Hence that $R^T$ is positive is equivalent to that $R^T_\alpha$ is positive for all $\alpha$. Moreover, a Sasaki-Ricci flow solution on $M$ induces a K\"ahler-Ricci flow solution on $V_\alpha$ for each $\alpha$. With this relation we can see that the positivity of (transverse) holomorphic bisectional curvature is preserved under the Sasaki-Ricci flow, by extending the results for K\"ahler-Ricci flow.

For $3$-dimensional  compact  K\"ahler manifolds, S. Bando \cite{Bando} proved that the positivity of holomorphic bisectional curvature is preserved along the K\"ahler-Ricci flow. This was later proved by N. Mok \cite{Mok} for all dimensions. As a direct consequence of their results, we can get the same property of transverse holomorphic bisectional curvature along the Sasaki-Ricci flow. 

We consider the evolution equation of $R^T$ along the Sasaki-Ricci flow, and we denote
\begin{equation}\label{E-8-2}
\frac{\p R^T}{\p t}= \t^T R^T+F(R^T)+R^T,
\end{equation}
where in the local transverse holomorphic coordinates, 
\begin{equation}\label{E-8-3}
\begin{split}
F(R^T)_{\alpha\bar \alpha\beta\bar \beta}=\sum_{\mu, \nu}R^T_{\alpha\bar \alpha \mu\bar \nu}R^T_{\nu \bar \mu \beta \bar \beta}&-\sum_{\mu, \nu}\left|R^T_{\alpha\bar \mu\beta\bar \nu}\right|^2+\sum_{\mu, \nu}\left|R^T_{\alpha\bar \beta \mu\bar \nu}\right|^2\\
&-\sum_{\mu, \nu}Re\left(R^T_{\alpha\bar \mu}R^T_{\mu\bar\alpha\beta\bar\beta}+R^T_{\beta\bar\mu}R^T_{\alpha\bar \alpha\mu\bar\beta}\right).
\end{split}
\end{equation}

Since $g^T$ is a genuine K\"ahler metric on $V_\alpha$ and it is evolved by the K\"ahler-Ricci flow, then \eqref{E-8-2} follows from the standard computations in K\"ahler-Ricci flow. 
For any tensor $S$ which has the same type as $R^T$, we can also define $F(S)$ as in \eqref{E-8-3}. As in K\"ahler case,  $F(S)$ satisfies the following property, called the null vector condition. 
\begin{prop}[Bando; Mok]\label{P-8-2}
If $S \geq 0$ and there exist two nonzero vectors $X, Y\in \cD^{1, 0}$ such that $S_p(X, \bar X; Y, \bar Y)=0$, then $F_p(S)(X, \bar X; Y, \bar Y)\geq 0$.
\end{prop}

\begin{proof}Note that this is purely local problem and so we can deal with this problem for transverse metric $g^T_\alpha$ on local coordinates $V_\alpha$. Let $p\in U_\alpha\subset M$ such that $\pi(p)\in V_\alpha$, where $U_\alpha$ and $V_\alpha$ are local coordinates introduced in Section 3. Then $g^T_\alpha$ is evolved by K\"ahler-Ricci flow on $V_\alpha$.
Hence $F(R^T)$  restricted on $V_\alpha$ is exactly the same as in K\"ahler case, hence has the formula as in \eqref{E-8-3}; see \cite{Mok} (6) for example. Then the proposition follows from the results of Bando \cite{Bando} (complex dimension three) and Mok \cite{Mok} (all dimensions) in K\"ahler case, see Section 1.2 in \cite{Mok} for details. 
\end{proof}

Then we can get that 
 
 \begin{prop}\label{P-8-4}Along the Sasaki-Ricci flow, if the initial metric has nonnegative transverse holomorphic bisectional curvature, then the evolved metrics also have nonnegative transverse holomorphic bisectional curvature; if the initial transverse holomorphic bisectional curvature is positive somewhere, then the transverse holomorphic bisectional curvature is positive for $t>0$. 
\end{prop}

\begin{proof}With Proposition \ref{P-8-2}, the result follows similarly as in Proposition 1 in \cite{Bando}, where Hamilton's maximum principle for tensors was used in K\"ahler setting.  The only slight difference is that the auxiliary function should be basic in Sasaki case.  We shall  sketch the proof as follows. 

It is sufficient to prove the proposition for a short time, thus we consider it in a short closed interval without specification. In particular, all metrics evolved have bounded geometry, namely, they are all equivalent and have bounded curvature (and higher derivatives are bounded also). First we define a parallel tensor field $S_0$ as in local transverse holomorphic coordinates
\[
{S_0}_{i\bar jk\bar l}= \frac{1}{2}\left(g^T_{i\bar j}g^T_{k\bar l}+g^T_{i\bar l}g^T_{k\bar j}\right),
\]
where $S_0$ has the same type as $R^T$ and $S_0>0$ everywhere in the sense of Definition \ref{D-8-2}. Then there exists a positive constant $C\in \R$ such that
\[
-CS_0\leq \frac{d S_0}{dt}\leq CS_0. 
\]

When $f$ is a basic function, then $fS_0$ also has the same type as $R^T$ and so $F(R^T+fS_0)$ is well defined. It is clear that $F(S)$ is smooth for $S$, hence there exists a positive constant $D$ such that
\[
F(R^T)\geq F(R^T+ fS_0)-D|f|S_0, |f|\leq 1. 
\] 
Now we consider the tensor $R^T+\epsilon fS_0$ for some small positive number $\epsilon$ and time dependent function $f(t, x)$. We compute, if $f$ is basic for all $t$, 
\begin{equation}\label{E-8-4}
\begin{split}
\frac{\p }{\p t } (R^T+\epsilon_0 fS_0)=&\t R^T+F(R^T)+R^T+\epsilon \frac{\p }{\p t}S_0+\epsilon f \frac{\p S_0}{\p t}\\
=&\t (R^T+\epsilon f S_0)+F(R^T+\epsilon fS_0)+R^T+\epsilon fS_0\\
&+\epsilon \left(\frac{\p f}{\p t}-\t f\right)S_0+F(R^T)-F(R^T+\epsilon fS_0)\\
&+\epsilon f\left(\frac{\p S_0}{\p t}-S_0\right)\\
\geq & \t (R^T+\epsilon f S_0)+F(R^T+\epsilon fS_0)+R^T+\epsilon fS_0\\
&+\epsilon S_0\left(\frac{\p f}{\p t}-\t f-(C+D+1) f\right).
\end{split}
\end{equation}

Let $f(0, \cdot)\equiv 1$ and let $f(t, x)$ satisfy the equation
\[
\frac{\p f}{\p t}-\t f-(C+D+1) f=1. 
\]
In local coordinates, we can write
\[
\t f= \xi^2 f+g^{i\bar j}_T\p_i\p_{\bar j} f. 
\]
Clearly, $\xi g^T_{i\bar j}=0$ and $\xi$ can commute with $\p_i, \p_{\bar j}$ when taking derivatives. So we have
\[
\xi (\t f)=\t (\xi f).
\]
Hence $\xi f$ satisfy the equation
\[
\frac{\p (\xi f)}{\p t}=\t (\xi f)- (C+D+1) \xi f.
\]
By the maximum principle, $\xi f\equiv 0$ since $\xi f=0$ at $t=0$. Hence \eqref{E-8-4} is justified. Moreover it is clear $f>0$ for $t>0$. 
Since $R^T\geq 0$ at $t=0$, $R^T+\epsilon fS_0>0$ at $t=0$. Now we claim $R^T+\epsilon fS_0>0$ for all $t$ and small $\epsilon$. If not, there is a first time $t_0$ such that $R^T+\epsilon fS_0(X, \bar X, Y, \bar Y)=0$ at some point $(t_0, p)$ for some nonzero vectors $X, Y\in \cD^{1, 0}$ and $R^T+\epsilon fS_0>0$ for $t<t_0$. 
We now consider the problem locally on $V_\alpha$ and let $X_\alpha=\pi_{*} (X), Y_\alpha=\pi_{*}(Y)\in T^{1, 0} V_\alpha$. 
We can extend $X_\alpha, Y_\alpha$ as follows. At $t=t_0$, we extend $X_\alpha, Y_\alpha$ to a normal  neighborhood of $(t_0, \pi(p))$ in $[0, t_0]\times V_\alpha$  by parallel transformation along radial geodesics of $g^T_\alpha$ at $(t_0, \pi(p))$ such that $\nabla X_\alpha=\nabla Y_\alpha=0$ at $(t_0, \pi(p))$, and then we extend $X_\alpha, Y_\alpha$ around $t_0$ such that $\p_t X_\alpha, \p_tY_\alpha=0$. We can then compute, at $(t_0, p)$, 
\[
\frac{\p }{\p t} (R^T+\epsilon fS_0) (X_\alpha, \bar X_\alpha, Y_\alpha, \bar Y_\alpha)=\left(\frac{\p }{\p t} (R^T+\epsilon fS_0) \right) (X_\alpha, \bar X_\alpha, Y_\alpha, \bar Y_\alpha)\leq 0,
\]
and 
\[
0\leq \t \left((R^T+\epsilon fS_0)(X_\alpha, \bar X_\alpha, Y_\alpha, \bar Y_\alpha)\right)= (\t (R^T+\epsilon f S_0)) (X_\alpha, \bar X_\alpha, Y_\alpha, \bar Y_\alpha).
\]
But by \eqref{E-8-4}, we get
\[
\begin{split}
0&\geq \left(\frac{\p }{\p t} (R^T+\epsilon fS_0) \right) (X_\alpha, \bar X_\alpha, Y_\alpha, \bar Y_\alpha)\\
 &\geq (\t (R^T+\epsilon f S_0)) (X_\alpha, \bar X_\alpha, Y_\alpha,\bar Y_\alpha)+\epsilon S_0(X_\alpha, \bar X_\alpha, Y_\alpha, \bar Y_\alpha)>0,
\end{split}
\]
since $F(R^T+\epsilon fS_0)$ satisfies the null vector condition and $(R^T+\epsilon fS_0)(X, \bar X, Y, \bar Y)=0$. Contradiction. 
Now let $\epsilon \rightarrow 0$, we prove that $R^T\geq 0$ for all $t$. 

The proof for the last statement of the proposition is also similar. But we shall choose the function $f$ more carefully. Recall $R^T\geq 0$ and $R^T$ is positive somewhere, at $t=0$. We can define a function $f_0(p)=\min R^T_p(X, \bar X, Y, \bar Y)$ for $p\in M, X, Y\in \cD^{1, 0}_p,  |X|=|Y|=1$. It is clear $f_0$ is nonnegative and cannot be identically zero. Since $\xi$ acts on $g$ isometrically,  $f_0$ is constant along any orbit of $\xi$. Hence $\xi f_0=0$ is well defined. Similarly we compute 
\[
\frac{\p }{\p t} (R^T-fS_0)\geq \t (R^T-fS_0)+F(R^T-fS_0)+\left(-\frac{\p f}{\p t}+\t f-(C+D)|f|\right)S_0.\]

We choose  $f(0, \cdot)=f_0$ such that
\[
\frac{\p f}{\p t}=\t f-(C+D) f.
\]
Note that $\xi f\equiv 0$ since $\xi f_0=0$. 
By our choice of $f$,  $R^T-fS_0\geq 0$ at $t=0$ and by the similar argument as above, we can get that $R^T-fS_0\geq 0$. If we let $\tilde f=e^{-(C+D)t}f$, then we can get that
\[
\frac{\p \tilde f}{\p t}=\t \tilde f. 
\]
We then get  $\tilde f>0$ for $t>0$, hence $f(t)>0$ when $t>0$. It follows that  $R^T>0$ for $t>0$. This completes the proof. 
\end{proof}

Then we study the Sasaki-Ricci flow for the metric with transverse holomorphic bisectional curvature. We have
\begin{theo}\label{T-8-1}Suppose $(M, \xi, g)$ is a Sasakian structure such that $g$ has nonnegative transverse holomorphic bisectional curvature then the Sasaki-Ricci flow with initial metric $g$ exists for all time with bounded curvature and it converges to a Sasaki-Ricci soliton $(M, \xi_\infty, g_\infty)$ subsequently.  \end{theo}

\begin{proof}Since  the transverse holomorphic bisection curvature becomes positive along the Sasaki-Ricci flow, then the transverse holomorphic bisectional curvature is then bounded by  its transverse scalar curvature, which is bounded along the flow by Theorem 7.1. 
Hence the transverse sectional curvature defined by $g^T$ is  bounded. 
It follows that  the sectional curvature of $g$ for any two unit vectors $X, Y\in \cD$ is then bounded by \eqref{E-2-curvature}. Moreover, by the definition, the sectional curvature of $g$ for any two plane in $TM$ containing $\xi$ is $1$, hence the sectional curvature of $g$ is uniformly bounded along the Sasaki-Ricci flow. Moreover the diameter of $M$ is  uniformly bounded and  the volume is fixed along the flow. It is well known that the Sobolev constants and the injectivity radii are uniformly bounded along the flow. 
Hence for any $t_i\rightarrow \infty$, there is a subsequence such that $(M,  g(t_i))$ converges to a compact  manifold $(M, \xi_\infty, g_\infty)$ in Cheeger-Gromov sense. Namely there is a sequence of diffeomorphisms $\Psi_i$ such that $\Psi_i^{*} g(t_i)$ converges to $g_\infty$, if necessarily, after passing to a subsequence. And ${\Psi_i}_{*}\xi$ and $\Psi_i^{*}\eta(t_i)$ converge to $\xi_\infty, \eta_\infty$ as tensor fields, which are compatible with $g_\infty$ such that $(M, \xi_\infty, g_\infty, \eta_\infty)$ defines a Sasakian structure.

To understand the structure of the limit metric, we recall  Hamilton's compactness theorem for Ricci flow, which applies to the Sasaki setting provided that  the curvature, the diameter and the volume are all uniformly bounded. For any $t_i\rightarrow \infty$, we consider the sequence of Sasakian-Ricci flow $(M, g(t_i+t))$. Then  after passing to some subsequence, it converges to the limit Sasaki-Ricci flow $(M, g_\infty(t)).$ As in K\"ahler setting \cite{Sesum}, one can prove that the limit flow is actually one parameter family of Sasakian-Ricci solitons evolved along the Sasaki-Ricci flow.  First observe that along the flow, the $\mu$ functional is uniformly bounded from above. 
Let $u(t)$ be the normalized (transverse) Ricci potential, then
\[
\mu(g(t), 1)\leq \int_M e^{-u}(R^T+u+|\nabla u|^2)d V_{g(t)}\leq C
\]
for some uniformly bounded constant $C$ since $R^T, |u|$ and $|\nabla u|$ are uniformly bounded along the flow. And $\mu(g(t), 1)$ is increasing, hence $\lim_{t\rightarrow \infty}\mu(g(t), 1)$ exists.   This implies that the $\mu$ functional for $(M, \xi_\infty, g_\infty(t))$ is a constant. To show $g_\infty(t)$ is a Sasakian-Ricci soliton,  first  we assume that  for some $T\in (0, \infty)$, $\mu(g_\infty(T), 1)$ has a smooth positive minimizer $w_0$ (we are using the form of \eqref{E-3-22}), then consider the backward heat equation for $w_0$ such that
\[
\frac{dw(t)}{dt}=-\t w+(R^T-n)w, w(T)=w_0. 
\]
As in Proposition \ref{P-3-5}, $w(t)$ is smooth and positive. In particular, $\cW(g_\infty(t), w(t), 1)$ is increasing for $t\in [0, T]$. Hence we get  $\mu(g_\infty(t), 1)\leq \cW(g_\infty(t), w(t), 1)\leq \mu(g_\infty(T), 1)$. But $\mu(g_\infty(t), 1)$ is constant along the flow. This implies that  $w(t)$ minimizes $\cW(g_\infty, w(t), 1)$. In particular, let $f (t)=-log w(t)$, 
then we have
\begin{equation}\label{E-8-5}
\frac{d }{d t}\cW(g_\infty(t), f(t), 1)\equiv0
\end{equation}
for $t\in [0, T]$. 
By Proposition \ref{P-3-4}, we get that, for any $t\in [0, T]$,
\[{R^T_\infty}_{i\bar j}+f_{i\bar j}-{g^T_\infty}_{i\bar j}=0, f_{ij}=0.\] 

In general, let $w_0$ be a nonnegative  minimizer in $W^{1, 2}_B$ for $\mu(g_\infty(T), 1)$. 
Let $w(t)$ be the solution of the backward heat equation
\[
\frac{\p w}{\p t}=-\t w+ (R^T-n\tau^{-1}) w
\]
such that $w(T)=w_0$. The standard parabolic regularity theory implies that $w(t)\in C^\infty$. By  Proposition \ref{P-3-5}, $w(t)\geq 0$ for $t\in [0, T)$. 
It is clear that  $w(t)$ can never be identically zero, hence  $w(t)>0$ for all $t\in [0, T)$. 
Note we can repeat the argument above to get that  $w(t)$ minimizes $\mu(g_\infty(t), 1)$ and  that $(M, g_\infty(t))$ is a Sasaki-Ricci soliton for any $t\in [0, T)$. Hence $(M, g_\infty(t))$ is a Sasaki-Ricci soliton for any $t\in [0, \infty)$. 
 \end{proof}

\begin{rmk} On Riemannian surfaces, there exists nontrivial Ricci soliton with positive curvature on sphere with orbifold singularity \cite{Wu, CW}. Note that this corresponds to the case of  quasi-regular Sasakian metric on wighted 3 sphere with positive (transverse) bisectional curvature, in particular the limit soliton does not have to be transverse Kahler-Einstein. It would be very interesting to classify Sasaki-Ricci soliton with positive transverse bisectional curvature. I am grateful to Xiuxiong Chen and Song Sun for valuable discussions on this topic.
\end{rmk}

\section{Appendix}
We summarize some results for Sakakian manifolds with the positive transverse bisectional curvature.  The topology of compact Sasakian manifolds with positive curvature in suitable sense are  well studied, for example see \cite{Goldberg, Moskal, Tachibana-Ogawa, Tanno, BGN}.  

First we need several formulas about basic forms on Sasakian manifolds. All these formulas are  similar as as the corresponding formulas in the K\"ahler setting, see Morrow-Kodaira \cite{MK} for the corresponding K\"ahler version.

\begin{prop}Let $\phi=\phi_{A\bar B}dz_A\wedge dz_{\bar B}$ be a basic $(p, q)$ form on a Sasakian manifold $(M, \xi, g)$. 
Then
\[
\begin{split}
\p \phi=\nabla_i^T \phi_{A\bar B}dz_i\wedge dz_A\wedge dz_{\bar B},\\
\bar \p \phi=\nabla_{\bar j}^T\phi_{A\bar B} dz_{\bar j}\wedge dz_A\wedge dz_{\bar B}. 
\end{split}
\]
\end{prop}

\begin{proof}
We have
\[
\nabla_i^T \phi_{A\bar B}=\p_i\phi_{A\bar B}-\Gamma_{i\alpha_k}^\sigma \phi_{\alpha_1\cdots \alpha_{k-1} \sigma \alpha_{k+1}\cdots \alpha_p\bar B}
\]
Note that $\Gamma^\sigma_{i\alpha_k}=\Gamma^\sigma_{\alpha_k i}$ ($\Gamma$ is symmetric on $i$ and $\alpha_k$ ), while $dz_i\wedge dz_A\wedge dz_{\bar B}$ is skew-symmetric on $dz_i$ and $dz_{\alpha_k}$. Hence
\[
\nabla_i^T \phi_{A\bar B}dz_i\wedge dz_A\wedge dz_{\bar B}=\p_i\phi_{A\bar B}dz_i\wedge dz_A\wedge dz_{\bar B}=\p \phi. 
\]
Similarly we can get the formula for $\bar \p \phi.$
\end{proof}

\begin{prop}Let $(M, \xi, g)$ be a compact  Sasakian manifold. Then 
\begin{equation}\label{E-9-i1}
\begin{split}
(\bar \p_B \phi, \psi)&=(\phi, \bar \p^{*}_B\psi), \mbox{for}\; \phi\in \Omega^{p, q-1}_B, \psi\in \Omega^{p, q}_B,\\
(\p_B \phi, \psi)&=(\phi, \p^{*}_B\psi),  \mbox{for}\; \phi\in \Omega^{p-1, q}_B, \psi\in\Omega^{p, q}_B,\\ 
(d_B\phi, \psi)&=(\phi, \delta_B \psi), \mbox{for}\; \phi\in \Omega^{r-1}_B, \psi\in \Omega^{r}_B. 
\end{split}
\end{equation}
In particular, we have
\[
\bar\p^{*}_B\psi_{A\bar B}=- (-1)^p g^{i\bar j}_T\nabla_i\psi_{A\bar j \bar B},
\]
where $\psi\in \Omega^{p, q+1}_B$. 
\end{prop}
\begin{proof}The formulas in \eqref{E-9-i1} are well known, see \cite{KT87} for example (the authors deal with only real forms in \cite{KT87}. But it is a straightforward extension to complex forms). For the sake of completeness, we include a proof of the first identity in \eqref{E-9-i1}, see \cite{MK} for the K\"ahler setting for example. By definition
\[
(\bar \p_B \phi, \psi)=\int_M \bar \p_B\phi \wedge *_B \bar \psi \wedge \eta. 
\]
Let $\Phi= \phi\wedge *_B \bar \psi\in \Omega^{n, n-1}_B$. Hence $\p_B\Phi=0$ and $d\Phi=d_B\Phi=\bar \p_B\Phi$. Note that $d\eta=\i g^{T}_{i\bar j}dz_i\wedge dz_{\bar j}\in \Omega^{1, 1}_B$.
We have
\[
d(\Phi\wedge \eta)=-\Phi\wedge d \eta+d_B\Phi\wedge \eta=\bar \p_B\Phi\wedge \eta. 
\] 
It follows that
\[
0=\int_M\bar \p_B \Phi\wedge \eta=\int_M \bar \p_B \phi\wedge *_B \bar \psi\wedge \eta+(-1)^{p+q-1}\int_M \phi\wedge \bar\p_B *_B \bar\psi\wedge \eta.
\]
Hence
\[
\begin{split}
(\bar\p_B \phi, \psi)&=(-1)^{p+q}\int_M \phi\wedge \bar\p_B *_B \bar\psi\wedge \eta\\
&=(-1)^{p+q} \int_M\phi \wedge (-1)^{2n+1-p-q}  *_B*_B\bar\p_B*_B \bar\psi\wedge \eta\\
&=\int_M \phi\wedge *_B \overline{(-1)*_B\p_B *_B \psi}\wedge \eta\\
&=(\phi, \bar\p^{*}\psi). 
\end{split}
\]
Note that $*_B$ can be characterized by \[\phi\wedge *_B\bar \phi=\frac{(d\eta)^n}{2^nn!}=\det(g^T_{i\bar j})dZ\wedge d\bar Z\] for basic forms. Then we have, for basic $(p, q)$ forms $\phi=\phi_{A\bar B} dz_A\wedge dz_{\bar B}$ and $\psi=\psi_{C\bar D}dz_C\wedge dz_{\bar D}$, 
\begin{equation}\label{E-9-i2}
(\phi, \psi)=\int_M \phi_{A\bar B}\overline{ \psi_{C\bar D}}g^{A\bar C}_T g^{B\bar D}_T dvol_g. 
\end{equation}
To compute $\bar \p^{*}_B \psi$, let $\phi=\phi_{A\bar B}dz_A\wedge dz_{\bar B}$ be a basic $(p, q)$ form and $\psi=\psi_{C\bar D_0}dz_C\wedge dz_{\bar D_0} $ be a basic $(p, q+1)$ form. We denote $D_0=iD$ and $B_0=jB$. 
Then we have, 
\[
\begin{split}
(\bar\p_B \phi, \psi)&=\int_M (-1)^p\nabla^T_{\bar j}\phi_{A\bar B} \overline{\psi_{C\bar D_0}}g_T^{A\bar C}g_{T}^{D_0\bar B_0} dvol_g\\
&=-(-1)^p\int_M\phi_{A\bar B}\overline{\nabla^T_j \psi_{C\bar i\bar D}} g_T^{A\bar C}g_{T}^{D\bar B}g^{i\bar j}_Tdvol_g\\
&=(\phi, \bar \p^*_B\psi).  \end{split}
\]
Using \eqref{E-9-i2} and the above, we can get that
\[
\bar \p^*_B\psi_{C\bar D}= -(-1)^p g^{j\bar i}_T\nabla_j\psi_{C\bar i\bar D}.
\]

\end{proof}

Then we have the following, 

\begin{prop}\label{P-9-1}Let $\Phi=\Phi_{A\bar B} dz_A\wedge dz_{\bar B}$ be a basic $(p, q )$ form. Then we have 
\begin{equation}\label{E-9-i3}
\t_{\bar \p}\Phi_{A\bar B}=-g^{i\bar j}_T \nabla_i^T\nabla_{\bar j}^T \Phi_{A\bar B}-\sum_{k}(-1)^k g^{i\bar j}_T[\nabla_i^T, \nabla^T_{\bar {\beta}_k}]\Phi_{A\bar j\bar \beta_1\cdots {\hat {\bar \beta}}_k\cdots \bar \beta_q}.
\end{equation}
Moreover, when $\Phi=\Phi_{\bar B}dz_{\bar B}$ is a basic $(0, q)$ form, we have
\begin{equation}\label{E-9-i4}
\t_{\bar \p}\Phi_{\bar B}=-g^{i\bar j}_T\nabla_i^T\nabla^{T}_{\bar j}\Phi_{\bar B}+\sum_{k}R^T_{\tau \bar\beta_k}\Phi_{\bar \beta_{1}\cdots \bar \beta_{k-1}\bar \tau \bar \beta_{k+1}\cdots \bar \beta_{q}}
\end{equation}
\end{prop}
\begin{proof}Recall
 $\t_{\bar \p}=\bar \p_B\bar \p^*_B+\bar \p^*_B\bar \p_B$. \eqref{E-9-i3} follows from the direct computation with Proposition 9.1 and 9.2 and \eqref{E-9-i4} is a direct consequence of \eqref{E-9-i3}. Note that  we can consider this problem locally involved with transverse K\"ahler metric $g^T_\alpha$ on $V_\alpha$. By Proposition 9.1 and 9.2, $\bar \p_B, \bar \p^*_B$ can be viewed as the corresponding operator of the K\"ahler metric $g^T_\alpha$ on $V_\alpha$. Hence the computation of $\t_{\bar \p}$ on basic forms has the same local formula as the corresponding formula in the K\"ahler setting, with the K\"ahler metric replaced by the transverse K\"ahler metric (see \cite{MK} Chapter 2, Theorem 6.1 for the K\"ahler setting; note that there is a sign difference in \cite{MK} for $R_{i\bar j}$ with ours).
\end{proof}

\begin{prop}\label{P-9-01}
Let $(M, g)$ be a compact Sasakian manifold of dimension $(2n+1)$ such that $Ric^T\geq 0$ and $Ric^T$ is positive at one point. Then
\[
b_1(M)=b_1^B=0. 
\]
Moreover, we have the following transverse vanishing theorem
\[
H^{q, 0}(\cF_\xi)=H^{0, q}(\cF_\xi)=0, q=1, \cdots, n.
\]
\end{prop}
\begin{proof} The statement $b_1(M)=0$ is proved by Tanno with the assumption  $Ric+2g>0$ (see \cite{Tanno} Theorem 3.4). Given that $Ric^T(X, Y)=Ric(X, Y)+2g(X, Y)$ for $X, Y\in \cD$, $Ric(\xi, \xi)=2n$, $Ric(\xi, X)=0$, $X\in \cD$,  then $Ric^T>0$ is equivalent to $Ric+2g>0$. When $(M, g)$ is assumed to be positive Sasakian structure, Boyer-Galicki-Nakamaye \cite{BGN} proved $b_1=0$, and also proved the vanishing theorem when $(M, g)$ is quasi-regular. 
We would use the transverse Hodge theory to  prove that $H^1_B(\cF_\xi)=0$ and we know that $H^1(M, \R)\approx H^1_B(\cF_\xi)$ (see \cite{BG}, page 215 for example).
Suppose $w=w_{\bar l}dz_{\bar l}$ is a basic harmonic $(0, 1)$ form. Then by Proposition \ref{P-9-1} we have
\[
\t_{\bar \p} w_{\bar l}=-g^{i\bar j}_T \nabla_i^T\nabla_{\bar j}^T w_{\bar l}+R^T_{p\bar l} w_{\bar p}=-g^{i\bar j}_T\nabla_{\bar j}^T \nabla_i^Tw_{\bar l}=0. 
\]
It then follows that
\[
\int_M R^T_{p\bar l} w_{\bar p} \bar w_{\bar l}-g^{i\bar j}_T \nabla_i^T\nabla_{\bar j}^T w_{\bar l}\bar  w_{\bar l}=-\int_M g^{i\bar j}_T\nabla_{\bar j}^T \nabla_i^Tw_{\bar l} \bar w_{\bar l}=0.
\]
Hence
\[
\int_M R^T_{p\bar l} w_{\bar p}\bar  w_{\bar l}+|\nabla^T_{\bar j} w_{\bar l}|^2=\int_M |\nabla^T_i w_{\bar l}|^2= 0. 
\]
This implies that $\nabla^T w_{\bar l}=0$ and $R^T_{p\bar l} w_{\bar p}\bar  w_{\bar l}=0$. Since $Ric^T\geq 0$ and it is positive at some point $p$, it follows that $w_{\bar l}(p)=0$. Hence $w_{\bar l}\equiv 0$. It follows that 
$b^{0, 1}_B=b^{1, 0}_B=0$. Hence $b^1_B=0$. Now we prove  $b^{q, 0}_B=b^{0, q}_B$  $q> 1$ in the similar way. Suppose $\Phi=\Phi_{\bar B}dz_{\bar B}$ is a harmonic $(0, q)$ form.
Then  by Proposition \ref{P-9-1} we have,
\[
\t_{\bar \p}\Phi_{\bar B}=-g^{i\bar j}_T\nabla_i^T\nabla^{T}_{\bar j}\Phi_{\bar B}+\sum_{k}R^T_{\tau \bar\beta_k}\Phi_{\bar \beta_{1}\cdots \bar \beta_{k-1}\bar \tau \bar \beta_{k+1}\cdots \bar \beta_{q}}=-g^{i\bar j}_T\nabla_{\bar j}^T\nabla_{i}^T\Phi_{\bar B}=0. 
\]
It then follows that
\[
\int_M R_{\tau \bar \beta_{k}} \Phi_{\bar \beta_{1}\cdots \bar \beta_{k-1}\bar \tau \bar \beta_{k+1}\cdots \bar \beta_{q}}\overline{\Phi_{\bar \beta_1\cdots \bar \beta_q}}+|\nabla^T_{\bar j}\Phi_{\bar B}|^2=\int_M |\nabla_i^T\Phi_{\bar B}|^2=0
\]
By the positivity  assumption on $Ric^T$, we can get that $\Phi_{\bar B}\equiv 0$. 
\end{proof}

\begin{prop}\label{P-9-4}Let $(M, g)$ be a compact Sasakian manifold of dimension $(2n+1)$ such that the transverse bisectional curvature $R^T\geq 0$ and it is positive at one point. Then
\[
b_2(M)=0, b_2^B(\cF_\xi)=1. 
\]
\end{prop}
\begin{proof}First we have $b_1(M)=b^B_1=0$. Note that we have the following exact sequence (see \cite{BG}, page 215 for example)
\[
0=H^1(M, \R)\rightarrow \R\rightarrow H^{2}_B(\cF_\xi)\rightarrow H^2(M, \R)\rightarrow H^1_B(\cF_\xi)=0.
\]
It follows that $b_2(M)=b^B_2-1$. Now we shall use the transverse Hodge theory to show that $b_2^B=1$ if the transverse bisectional curvature is positive. By Proposition \ref{P-9-01}, we have $b^{2, 0}_B=b^{0, 2}_B=0$. Now we prove $b^{1, 1}_B=1$ (see \cite{GK} for the K\"ahler setting). Let $\psi=\psi_{k\bar l} dz_k\wedge dz_{\bar l}$ be a basic harmonic $(1, 1)$ form. Then by Proposition \ref{P-9-1} we have
\[
\begin{split}
\t_{\bar \p} \psi_{k\bar l}&=-g^{i\bar j}_T \nabla^T_{i}\nabla_{\bar j}^T \psi_{k\bar l}-g^{i\bar j}_TR^T_{i\bar q k\bar l}\psi_{q\bar j}+R^T_{p\bar l}\psi_{k\bar p}\\
&=-g^{i\bar j}_T\nabla^T_{\bar j}\nabla^T_i \psi_{k\bar l}-g^{i\bar j}_TR^T_{i\bar q k\bar l}\psi_{q\bar j}+2R^T_{p\bar l}\phi_{k\bar p}-R^T_{k\bar p}\psi_{p\bar l}
\end{split}
\]
Choose an orthonormal frame such that $\psi_{k\bar l}=\delta_{kl}\psi_{k\bar k}$, then we get
\[
\left(-g^{i\bar j}_TR^T_{i\bar q k\bar l}\psi_{q\bar j}+R^T_{p\bar l}\psi_{k\bar p}\right) \overline{\psi_{k\bar l}}=\sum_{i< j} R_{i\bar i j\bar j}^T (\psi_{i\bar i}-\psi_{j\bar j})^2\geq 0,
\]
and 
\[
\left(-g^{i\bar j}_TR^T_{i\bar q k\bar l}\psi_{q\bar j}+2R^T_{p\bar l}\phi_{k\bar p}-R^T_{k\bar p}\psi_{p\bar l}\right) \overline{\psi_{k\bar l}}=\sum_{i< j} R_{i\bar i j\bar j}^T (\psi_{i\bar i}-\psi_{j\bar j})^2\geq 0.
\]
By the positivity assumption on $R^T_{i\bar j k \bar l}$, it then follows that $\nabla^T\psi\equiv 0$ and $\psi_{i\bar i}=\psi_{j\bar j}$ at some point $p$  for all $i, j$. Hence $\psi=\lambda d\eta$ for some constant $\lambda$. Hence
$b^{1, 1}_B=1$. 
\end{proof}

If  $(M, g)$ is Sasaki-Einstein and $(M, g)$ has positive transverse bisectional curvature, then $(M, g)$ has to be a space form with positive curvature $1$. 

\begin{prop}\label{P-9-5}
If $(M, g)$ is a Sasaki-Einstein metric such that the transverse K\"ahler structure has positive holomorphic bisectional curvature, then $(M, g)$ has constant sectional curvature $1$.  
\end{prop}
\begin{proof}

When $(M, g)$ is a K\"ahler manifold such that the bisectional curvature is positive and $g$ is K\"ahler-Einstein, then Berger \cite{Berger} and Goldberg-Kobayashi \cite{GK} proved that $g$ has constant bisectional curvature. 
In the Sasakian setting, if  $(M, g)$ is Sasaki-Einstein with positive sectional curvature, then $(M, g)$ has constant sectional curvature $1$ (c.f Moskal \cite{Moskal} and \cite{BG} Chapter 11). 

Here we follow Chen-Tian's argument  \cite{Chen-Tian} in K\"ahler setting  using the maximum principle to prove that the transverse K\"ahler metric $g^T$ has constant transverse bisectional curvature. 
 One  can easily get that the sectional curvature of $g$ is $1$ for a Sasaki-Einstein metric if  its transverse K\"ahler metric $g^T$ has constant transverse bisectional curvature,  
\[
R_{i\bar j k\bar l}^T=2(g_{i\bar j}^Tg_{k\bar l}^T+g_{i\bar l}^Tg_{k\bar j}^T). 
\]

Suppose that $g$ is a Sasaki-Einstein metric and  $g^T$ is transverse K\"ahler Einstein with positive transverse bisectional curvature. Then we can apply the argument of the maximum principle as 
in  \cite{Chen-Tian} (Lemma 8.19) to the Sasakian case to show  that $g^T$ has constant transverse bisectional curvature. The compactness of $M$ is only needed to find a maximum (or minimum)  of  auxiliary functions. 
After that we can do all  computations locally  on some $V_\alpha$.  Only  the transverse K\"ahler structure is involved.  Hence the argument in \cite{Chen-Tian} is applicable and the details to the Sasakian case can be carried out as in Proposition \ref{P-8-4}. We shall skip the details. 
It is clear that if we only assume that the transverse bisectional curvature is nonnegative and it is positive at one point, we can get the same conclusion. 

\end{proof}

With the result in \cite{Zhang}, one can actually prove 
\begin{prop}If $(M, g)$ is a Sasakian metric with constant scalar curvature and positive transverse bisectional curvature, then  $(M, g)$ has constant transverse bisectional curvature. In particular, there is a $D$-homothetic transformation such that $\tilde g=\alpha g+\alpha(\alpha-1)\eta\otimes \eta$ has constant sectional curvature $1$ for some $\alpha>0$. 
\end{prop}
\begin{proof}
In \cite{Zhang}, X. Zhang proved that $(M, g)$ is transverse K\"ahler-Einstein. We can actually give a simple proof of this fact as follows. Since $g$ has positive transverse bisectional curvature, by Proposition \ref{P-9-4}, $b_2^B(\cF_\xi)=1$. If the transverse scalar curvature $R^T$ is constant, 
we have
\[
\i \bar \p^{*} \rho^T=-[\Lambda, \p]\rho^T=\p (\Lambda \rho^T)=\p R^T=0,
\]
hence the transverse Ricci form $\rho^T$ is a harmonic  form, and it is a basic $(1, 1)$ form. If $b_2^B(\cF_\xi)=1$, then $\rho^T= R^T d\eta/2$, hence it is a transverse K\"ahler metric with constant transverse Ricci curvature. As in Proposition \ref{P-9-5}, one can further show that $g^T$ has constant transverse bisectional curvature. Since $R^T>0$ and $g^T$ is transverse K\"ahler-Einstein, it is well known that after a $D$-homothetic transformation such that $\tilde g=\alpha g+\alpha(\alpha-1)\eta\otimes \eta$ foe some $\alpha>0$, $\tilde g$ is Sasakian-Einstein. Since $\tilde g^T$ is just a rescaling of $g^T$, it has constant transverse bisectional curvature. Hence by Proposition \ref{P-9-5}, $\tilde g$ has constant sectional curvature $1$. It is clear that if we only assume that the transverse bisectional curvature is nonnegative and it is positive at one point, we can get the same conclusion. 
 \end{proof}

 We then discuss the minimizers of $\cW$ functional as in \eqref{E-3-22}. We finish the proof of  Theorem \ref{E-4-R}. 
 \begin{proof}The existence of a nonnegative minimizer is almost identical as in \cite{Rothaus}, Section 1.   To prove actually $w_0$ satisfies \eqref{E-w-1}, we follow \cite{Rothaus80, Rothaus} with a slight modification. 
We consider the function $L(\epsilon)$, as in \cite{Rothaus80}, for any $u\in W^{1, 2}_B$ fixed, 
\[
L(\epsilon)=L(w_0+\epsilon u)=\log \int_M (w_0+\epsilon u)^2 \tau^{-n}+\left(\int_M (w_0+\epsilon u)^2\tau^{-n}\right)^{-1} \cW(g, w_0+\epsilon u, \tau).
\]
Then it is straightforward to check that  $L(\epsilon )$ has a minimum at $\epsilon=0$. Taking derivative of $L(\epsilon)$ at $\epsilon=0$, then we get, for any $u\in W^{1, 2}_B$, 
\begin{equation}\label{E-w-2}
\int_M (w_0\log w_0^2-R^T w_0+\mu(g, \tau) w_0) u-4\langle \nabla w_0, \nabla u\rangle=0.
\end{equation}
If we assume $w_0$ is smooth (hence $w_0\in C^\infty_B$), then we get 
\[
\int_M (4\t w_0+w_0\log w_0^2-R^T w_0+\mu(g, \tau) w_0) u=0
\]
for any $u\in W^{1, 2}_B$, in particular, for any $u\in C^\infty_B$. 
Note that if $v\in C^\infty_B$ such that $\int_M uv=0$ for any $u\in C^\infty_B$, then $v=0$. 
Hence if $w_0$ is smooth,  then it follows that $w_0$ satisfies \eqref{E-w-1}. If $w_0$ is not smooth, we consider the equation
\[
\t h=1/4(w_0\log w_0^2-R^T w_0+\mu(g, \tau) w_0):=A_0.
\]
It is clear that there exists a unique solution (up to addition of a constant) since we can get, by taking $u=1$ in \eqref{E-w-2}, 
\[
\int_M w_0\log w_0^2-R^T w_0+\mu(g, \tau) w_0=0. 
\]
By an approximation argument, it is clear  that $h\in W^{1, 2}_B$ since $w_0\in W^{1, 2}_B$ and $R^T\in C^\infty_B$. 
Actually let $w_k\rightarrow w_0$ in $W^{1, 2}$ such that $w_k\in C^\infty_B$, then we would get a sequence of solution of $h_k\in C^\infty_B$ such that
\[
\t h_k=1/4(w_k\log w_k^2-R^T w_k+\mu(g, \tau) w_k):=A_k.
\]
By Sobolev inequality, we can assume that $A_k\rightarrow A_0$ in $L^p$, $p=2(2n+1)/(2n-1)-\epsilon$, for any small $\epsilon>0$. We can assume, for each $k$, 
\[
\int_M h=\int_M h_k=0. 
\]
 Then apply the standard elliptic regularity theory to $\t (h-h_k)=A_0-A_k$, we know that 
 \[
 \|h-h_k\|_{W^{2, p}}\leq C\|A_0-A_k\|_{L^p}. 
 \]
 In particular, $h_k\rightarrow h$ in $W^{2, p}$, hence $h\in W^{1, 2}_B$. Now using \eqref{E-w-2} again, we can get that for any $u\in C^\infty_B$, 
\[
\int_M (h-w_0)\t u=0. 
\]
We claim that $h-w_0$ is a constant. Let $u_k\in C^\infty_B$ be a sequence of smooth function such that $u_k\rightarrow h-w_0$ in $W^{1, 2}$. 
Then we get that
\[\int_M (h-w_0)\t u_k=0.\]
It then follows 
\[
\int_M\langle \nabla (h-w_0), \nabla u_k\rangle=0. 
\]
Hence we have, letting $k\rightarrow \infty$, 
\[
\int_M |\nabla (h-w_0)|^2=0. 
\]
Hence $h-w_0$ is a constant and $w_0$ satisfies
\begin{equation}\label{E-w-3}
4\t w_0=w_0\log w_0^2-R^T w_0+\mu(g, \tau) w_0.
\end{equation}
It then follows  the well-know De Giogi-Nash theory that $w_0\in L^\infty$, and then by  $L^p$ theory, we can get the further regularity $w_0\in W^{2, p}$ for any $p>1$. Furthermore, once we show that $w_0$ actually satisfies 
\eqref{E-w-3} almost everywhere, we can apply Lemma on page 114 (\cite{Rothaus}) to get that  $w_0>0$. Rothaus only stated this result  in \cite{Rothaus} for minimizers of $\eqref{E-3-22}$ in $W^{1, 2}$; but his proof only requires the equation \eqref{E-w-3}. Hence we can get $w_0>0$, and then it follows that $w_0$ is smooth. 
\end{proof}

\begin{rmk}In the present paper, we actually do not really need the fact that $w_0$ is positive and smooth. 
\end{rmk}

\bibliographystyle{amsplain}

\end{document}